\PassOptionsToPackage{dvipsnames}{xcolor}
\documentclass[]{siamart250211}

\usepackage{amssymb}
\usepackage{algpseudocode}
\usepackage{booktabs}
\usepackage[numbers, sort]{natbib}
\usepackage{mathbbol}
\usepackage{subfig}

\usepackage{tikz}
\usepackage{gnuplot-lua-tikz}
\usetikzlibrary{calc}

\newsiamremark{remark}{Remark}
\newsiamremark{assumption}{Assumption}

\renewcommand{\vec}[1]{\boldsymbol #1}

\newcommand{\domain}{D}
\newcommand{\bdry}{\partial D}
\newcommand{\diver}{\mathrm{div}}
\newcommand{\flux}{\mathbb{f}}
\newcommand{\normal}{\boldsymbol{n}}
\newcommand{\bzero}{\boldsymbol{0}}
\newcommand{\mom}{\boldsymbol{m}}

\newcommand{\vel}{\boldsymbol{v}}
\newcommand{\dt}{\tau}
\newcommand{\low}{\mathrm{L}}
\newcommand{\totme}{\mathcal{E}}
\newcommand{\state}{\vec u}
\newcommand{\vertices}{\mathcal{V}}
\newcommand{\transp}{\top}
\newcommand{\bc}{\mathbf{c}}
\newcommand{\Epot}{\varphi}
\newcommand{\xcoord}{\vec x}
\newcommand{\inte}{\varepsilon}
\newcommand{\specinte}{e}

\newcommand{\Ltwo}{L^2(\domain)}
\newcommand{\Ltwod}{L^2(\domain)^d}

\title{Structure-preserving finite-element approximations of the magnetic\\
  Euler-Poisson equations}
\author{%
  Jordan Hoffart\footnotemark[1]
  \and Matthias~Maier\footnotemark[1]
  \and John~N.~Shadid\footnotemark[2]
  \and Ignacio~Tomas\footnotemark[3]
}

\begin{document}

\maketitle

\renewcommand{\thefootnote}{\fnsymbol{footnote}}

\footnotetext[1]{%
  Department of Mathematics, Texas A\&M University, 3368 TAMU,
  College Station, TX 77843, USA (\email{\{jordanhoffart, maier\}@tamu.edu})}

\footnotetext[2]{%
  Sandia National Laboratories$^{1}$, P.O. Box 5800, MS 1320, Albuquerque,
  NM 87185, USA (\email{jnshadi@sandia.gov}); and Department of Mathematics
  and Statistics, University of New Mexico, MSC01 1115, Albuquerque, NM
  87131, United States}

\footnotetext[3]{%
  Department of Mathematics and Statistics, Texas Tech University,
  2500 Broadway, Lubbock, TX 79409, USA (\email{igtomas@ttu.edu})}

\renewcommand{\thefootnote}{\arabic{footnote}}

\begin{abstract}
  We develop a structure-preserving numerical discretization for the
  electrostatic Euler-Poisson equations with a given magnetic field. The
  magnetic field extends the purely electrostatic Euler-Poisson system
  considered previously, and changes the dynamics by introducing new
  time-scales into the problem: the cyclotron and diocotron time-scales.
  Our focus in the present work is on the efficient solution of problems
  close to the magnetic-drift limit, also called
  $\textrm{E}\!\times\!\textrm{B}$-drift limit. Such a regime is
  characterized by the co-existence of slowly moving, smooth flows with
  very high-frequency oscillations, spanning timescales with a difference
  in excess of $\mathcal{O}(10^{12})$, making their numerical solution
  quite challenging. Our scheme oversteps the plasma and cyclotron scales,
  with a time-step size restricted only by a hyperbolic CFL condition.

  The scheme preserves positivity of the density, positivity of the
  internal energy, a minimum principle for the specific entropy, and a
  total energy balance. The scheme uses an operator splitting approach
  composed of two subsystems: the compressible Euler equations of gas
  dynamics, and a source system. The source system couples the
  electrostatic potential, momentum, and Lorentz force, thus incorporating
  electrostatic plasma and cyclotron motions. Because of the high-frequency
  phenomena it describes, the source system is discretized with an implicit
  time-stepping scheme. We use a PDE Schur complement approach for the
  linear algebra associated to the implicit source system. The Lorentz
  force is eliminated pointwise from the implicit system, so that only a
  single scalar, non-symmetric Poisson-like problem has to be solved per
  time step, which in turn enables the use of matrix-free linear algebra.
  Because of the efficiency of the matrix-free linear algebra, the implicit
  source solve costs about as much as a single explicit hyperbolic update,
  and accounts for less than a third of the total wall time.

  We illustrate the capability of the scheme by computing a diocotron
  instability and present growth rates that compare favorably with existing
  analytical results. The model, though a simplified version of the
  Euler-Maxwell system, represents a stepping stone toward electromagnetic
  solvers that are capable of working in the electrostatic and
  magnetic-drift limits as well as the hydrodynamic regime.
\end{abstract}




\section{Introduction}

The Euler-Poisson system models the dynamics of a single electric species
subject to electric fields. This model is one of the simplest models
describing electrically non-neutral plasmas \cite{Davidson2001}. The
Euler-Poisson model contains hydrodynamic time-scales as well as
electrostatic plasma oscillations. The most notable feature of this model is
the large separation between the hydrodynamical and plasma time-scales, see
\cite{eulerpoisson23}. For reference, traditional MHD models are usually
derived as zero-permittivity
and zero-electron-mass limits of the two-fluid plasma model \cite{Goed2004}.
Therefore, traditional MHD models cannot possibly describe electrostatic
forces, plasma oscillations, and/or cyclotron motions, since such
phenomenology has been eliminated by construction \cite{Sentis2014}.

In the present paper, we explore the addition of a magnetic field to the
Euler-Poisson system. This is a subtle modification of the purely
electrostatic Euler-Poisson system considered in our previous work
\cite{eulerpoisson23}. The addition of the magnetic field to the
Euler-Poisson system introduces new time-scales and a new asymptotic limit.
In this paper, the main regime of interest is the
$\textrm{E}\!\times\!\textrm{B}$-drift limit, also called the guiding center
drift limit or the \emph{magnetic-drift limit}. The mathematical
understanding of \emph{kinetic models} and corresponding numerical schemes in
the magnetic-drift limit and the modeling of \emph{strongly magnetized
plasmas} is a topic that has received significant attention in the last two
decades \cite{Frenod1998, Golse1999, Frenod2000, Bostan2007, Mottez2008,
BenAbdallah2008, Herda2016, HanKwan2012, Frenod2015, Deluzet2017}. However,
the amount of published work related to the robust numerical approximation of
the magnetic-drift limit using \emph{fluid models} is rather scarce; we refer
to \cite{Degond2009, Degond2011} as notable exceptions.

For both kinetic and fluid models, the problem is hard enough such that a few
strong assumptions are usually made in order to make some analytical or
computational progress. For instance, standard assumptions are that the
$\textrm{B}$-field is perfectly perpendicular to the velocity field. Similarly,
it is common to assume that the electric field $\textrm{E}$ and/or the magnetic
field $\textrm{B}$ is \emph{given}, rather than computed self-consistently from
the evolution equations. In this paper we assume that the magnetic field
$\textrm{B}$ is given, but the electric field $\textrm{E}$ is computed
self-consistently from the evolution equations. This allows, in particular, to
focus on developing a computationally efficient time integration technique with
mathematical guarantees of structure preservation. Some basic questions of
practical interest are: What components of the system can be treated using
explicit time integration, and what parts require the use of implicit methods?
What structural properties need to be preserved?

Our goal with the present publication is to develop a semi-implicit
numerical method for the magnetic Euler-Poisson system that (a) is provably
\emph{robust}, meaning that it preserves relevant solution structure on the
\emph{fully discrete} level, and (b) allows to solve linear subproblems
efficiently without the need to resolve the fast time scales of
high-frequency electrostatic plasma and cyclotron oscillations. We note
that, at the time of this writing, the main approach explored for this
problem has consisted of fully coupled and fully implicit schemes
\cite{Crockatt2021, Crockatt2022, LaSpina2024}.

The outline of the paper is as follows: Section \ref{sec:preliminaries}
discusses preliminary aspects on the Euler equations, the Euler-Poisson system
with a given magnetic field and its time-scales, and the magnetic-drift limit.
Section \ref{sec:time-stepping} describes the operator splitting approach and
the PDE Schur complement technique (see Section \ref{sec:PDESchur}) used to
solve the corresponding linear system. Section \ref{sec:FullyDiscrete} provides
a fully discrete description of the scheme and discusses theoretical results.
Section \ref{sec:numerical} provides validation results and computations of the
diocotron instability with validation of growth rates. The current work is part
of the research program outlined by the authors in \cite{eulerpoisson23}, as
well as in two lab reports \cite{eulerpoisson21, eulerpoisson22}. Most of
the the new ingredients of the scheme are detailed in Section
\ref{sec:PDESchur}, which stem from the difficulties associated to magnetic
field and its elimination from the implicit system.


\section{Preliminaries}
\label{sec:preliminaries}

We first introduce the magnetic Euler-Poisson equations and summarize a number
of important structural properties of the PDE system. A detailed discussion of
the system and key features for the design of numerical methods can be found in
\cite{eulerpoisson21, eulerpoisson22, eulerpoisson23}. For the sake of
completeness, we summarize a number of key observations in detail that will be
important for the design of our numerical method.


\subsection{The compressible Euler equations with forces}
Let $\domain \subset \mathbb R^d$, $d=2,3$, be an open, bounded, simply
connected Lipschitz domain with unit outward normal $\vec n$. We consider a
fluid occupying $\domain$ with mass density $\rho(\vec x, t)$, velocity $\vec
v(\vec x, t) \in \mathbb R^d$, momentum (per unit volume) $\vec m := \rho\vec
v$, pressure $p(\vec x, t) \in \mathbb R$, and total mechanical energy (per
unit volume) $\mathcal E(\vec x, t)$. We assume that the fluid motion is given
by the compressible Euler equations subject to a force described by $\vec
f(\vec x, t) \in \mathbb R^d$ with units of force per volume. The compressible
Euler equations then read,
\begin{align}
  \label{eq:euler_force}
  \begin{cases}
    \begin{aligned}
      \partial_t \rho + \nabla \cdot \vec m
      &= 0,
      \\[0.5em]
      \partial_t \vec m + \nabla \cdot (\rho^{-1}\vec
      m\vec m^{\transp} + \mathbb I p)
      &= \vec f,
      \\[0.5em]
      \partial_t \mathcal E + \nabla \cdot (\rho^{-1}\vec m(\mathcal E
      + p)) & = \rho^{-1}\vec m \cdot
      \vec f.
    \end{aligned}
  \end{cases}
\end{align}
Let $\state := \big[\rho,\vec m, \mathcal E\big]^T$ denote the combined state
vector. We summarize a fundamental observation about the  time evolution of
entropies for such a forced system \cite{eulerpoisson21, eulerpoisson22,
eulerpoisson23, Dao2024}:
\begin{lemma}
  \label{lem:entropies_are_orthogonal}%
  Let $\Psi(\state):\mathbb{R}^{d+2} \rightarrow \mathbb{R}$ be a function
  of the state satisfying $\Psi(\state):= \psi\big(\rho,
  \specinte(\state)\big)$, where $\psi(\cdot, \cdot): \mathbb{R}^+ \times
  \mathbb{R}^+ \rightarrow \mathbb{R}$ and $\specinte(\state) :=
  \frac{\mathcal E}{\rho} - \frac{|\vec m|^2}{2 \rho^2}$ is the specific
  internal energy. Let $\state(\vec x, t)$ be a sufficiently smooth
  solution to \eqref{eq:euler_force}. Then,
  \begin{align}
    \label{eq:entropy_evolution}
    \partial_t\Psi(\state) +
    \nabla_{\state} \Psi(\state) \cdot
    \begin{bmatrix}
      \nabla \cdot \vec m
      \\
      \nabla \cdot (\rho^{-1}\vec m\vec m^{\transp} + \mathbb I p)
      \\
      \nabla \cdot (\rho^{-1}\vec m(\mathcal E + p))
    \end{bmatrix}
    = 0,
  \end{align}
  where $\nabla_{\state}$ is the gradient with respect to the state. In other
  words, the time evolution of $\Psi(\state(\vec x, t))$ does not directly
  depend on the force field $\vec f(\vec x, t)$.
\end{lemma}
\begin{proof}
  See \cite[Lemma\,2.1 and Remark\,2.2]{eulerpoisson23}.
\end{proof}
As an immediate consequence of the above lemma, the compressible Euler
equations with forces maintain the same entropy and invariant domain structure
as the system of equations without forces $\vec f$ \cite{eulerpoisson23,
Dao2024}:
\begin{corollary}
  \label{cor:invariant_domain}%
  Let $\mathcal B$ be an \emph{invariant domain}~\cite{GuePo2016I} of the
  compressible Euler equations that is fully described in terms of density
  $\rho$ and specific internal energy $\specinte$ of a state $\state$, such
  as
  \begin{align*}
    \mathcal B_\ast :=
    \Big\{
      \state\in\mathbb{R}^{d+2}\,:\,
      \rho > 0,\;
      \specinte(\state) > 0,\;
      s(\rho,\specinte(\state)) \ge s_{\mathrm{min}}
    \Big\},
  \end{align*}
  where $s(\rho,\specinte(\state))$ denotes the specific entropy. Then, the
  compressible Euler equations with forces, as described in
  \eqref{eq:euler_force}  maintains the same invariant domain.
\end{corollary}


\subsection{The Lorentz force}
\label{subse:lorentz}
We now assume that the fluid consists of a single electrically charged species;
for example, electrons with negative charge. The force $\vec f$ acting on the
fluid is then the \emph{Lorentz force} generated by an electric field $\vec
E(\vec x, t) \in \mathbb R^d$ and a magnetic flux density $\vec B(\vec x,t) \in
\mathbb R^{d_\ast}$:
\begin{align*}
  \vec f := \rho_{\mathrm e} \vec E + \rho_{\mathrm e} \vec v \times \vec B,
\end{align*}
where $\rho_{\mathrm e}(\vec x, t) := (q_{\mathrm{e}} / m_{\mathrm{e}})
\rho(\vec x, t)$ is the electric charge density, $q_{\mathrm{e}}$ is the
\emph{elemental charge}, and $m_{\mathrm{e}}$ the \emph{elemental mass} of the
species. For $d=3$ we set $d_\ast=3$, and, for $d=2$, we set $d_\ast=1$. For
$d=2$, by slight abuse of notation, we set the cross product to $\vec v \times
\vec B:=\big[v_2, -v_1\big]^T\,B$, which corresponds to a \emph{transverse
magnetic} configuration, where the fluid motion is confined to the 2D plane
spanned by $\hat{\vec e}_1$, $\hat{\vec e}_2$, and the magnetic field has only
a component in the $\hat{\vec e}_3$-direction.

The time evolution of the electric field $\vec E(\vec x, t)$ and the magnetic
flux density $\vec B(\vec x,t)$ is governed by Maxwell's equations, where the
electric charge density $\rho_{\mathrm e}(\vec x, t)$ and the electric current
density $(\rho_{\mathrm e}\vec v)(\vec x, t)$ enter as forcing terms
\cite{Bitten2004}. We make the assumption that the magnetic flux density is
constant in time, i.\,e., $\partial_t \vec B = \vec 0$, which implies by that
$\nabla \times \vec E = \vec 0$. From this and the assumption that $\domain$ is
simply connected, $\vec E = -\nabla \varphi_{\mathrm e}$ for some electric
potential $\varphi_{\mathrm e}(\vec x, t)$. Now, we introduce rescaled
quantities for the magnetic flux density and the electric potential, and we
define a coupling constant $\alpha$ \cite{eulerpoisson23},
\begin{align*}
  \varphi := \frac{q_{\mathrm e}}{m_{\mathrm e}}\varphi_{\mathrm e},
  \qquad
  \vec \Omega := \frac{q_{\mathrm e}}{m_{\mathrm e}}\vec B,
  \qquad
  \alpha := \frac{1}{\varepsilon_0}
  \frac{q_{\mathrm e}^2}{m_{\mathrm e}^2} \,>\, 0.
\end{align*}
Here, $\varepsilon_0$ denotes vacuum permittivity \cite{Chen1984}.
The definition of the parameter $\alpha$ motivates rewriting the PDE for the
electrostatic potential as:
\begin{align*}
  \partial_t\big(-\Delta\varphi\big) &= -\alpha\,\nabla\cdot\vec m.
\end{align*}
In summary, we consider the following Euler-Poisson model with a given
magnetic field:
\begin{definition}[Magnetic Euler-Poisson equations]
  Given a rescaled, static, divergence free magnetic flux density
  $\vec\Omega(\vec x)\in\mathbb{R}^{d_\ast}$, and initial data for the
  state $\state(\vec x,t)=[\rho,\vec m,\mathcal E]^T$ and the potential
  $\varphi(\vec x, t)$ satisfying the compatibility condition
  \begin{align}
    \label{eq:poisson_compatibility}
    -\Delta\varphi(\boldsymbol x, 0) = \alpha\,\rho(\boldsymbol x, 0),
  \end{align}
  find a solution to the evolution equation
  \begin{align}
    \label{eq:magnetic_euler_poisson}
    \begin{cases}
      \begin{aligned}
        \partial_t \rho + \nabla \cdot \vec m
        &= 0,
        \\[0.5em]
        \partial_t \vec m + \nabla \cdot (\rho^{-1}\vec
        m\vec m^{\transp} + \mathbb I p)
        &= -\rho\nabla\varphi + \vec m \times \vec\Omega,
        \\[0.5em]
        \partial_t \mathcal E + \nabla \cdot \big(\rho^{-1}\vec m(\mathcal E
        + p)\big) & = -\vec m \cdot \nabla\varphi,
        \\[0.5em]
        \partial_t(-\Delta\varphi) &= - \alpha\,\nabla \cdot \vec m,
      \end{aligned}
    \end{cases}
  \end{align}
  subject to appropriately chosen boundary conditions.
\end{definition}
Note that system \eqref{eq:magnetic_euler_poisson} has to be closed with an
\emph{equation of state} that determines the pressure $p(\vec
u)=p(\rho,\specinte(\vec u))$ as a function of the hydrodynamical state
$\state$. In the case of a barotropic closure, where the pressure depends only
on the density, the energy equation $\mathcal E$ is no longer necessary
\cite{FeireislBook}. The barotropic version of
\eqref{eq:magnetic_euler_poisson} is discussed with more detail in
Appendix~\ref{AppBaro}.

A number of remarks are in order:
\begin{remark}[Gauß law]
  \label{lem:gauss_law}%
  The magnetic Euler-Poisson equations \eqref{eq:magnetic_euler_poisson}
  formally maintain the Gauß law, i.\,e., if
\eqref{eq:poisson_compatibility}
  is maintained at the initial time, it holds true for all time.
More precisely, substituting  the first equation of
\eqref{eq:magnetic_euler_poisson} into  the fourth equation, we observe that
  \begin{align*}
    \partial_t(-\Delta\varphi - \alpha\rho)=0.
  \end{align*}
\end{remark}

\begin{lemma}[Energy balance]
  \label{lem:energy_balance}%
  The magnetic Euler-Poisson equations \eqref{eq:magnetic_euler_poisson}
  admit the following formal energy balance:
  \begin{multline}
    \label{eq:energy_balance}
    \frac{\mathrm d}{\mathrm dt}\int_\domain \mathcal E +
    \frac{1}{2\alpha}|\nabla\varphi|_{\ell^2}^2\,\mathrm dx
    \\
    +
    \int_{\bdry} \Big\{\rho^{-1}\vec m\big(\mathcal E + p\big) +
    \varphi\big(\vec m - \alpha^{-1}\partial_t \nabla \varphi\big)
    \Big\}\cdot\vec n\,\mathrm
    do_x
    = 0,
  \end{multline}
  where $\vec n$ denotes the outwards pointing unit normal on $\bdry$.
\end{lemma}
\begin{proof}
  Test the last equation of \eqref{eq:magnetic_euler_poisson} with
  $\alpha^{-1}\varphi$ and integrate by parts:
  \begin{equation*}
    \frac{\mathrm d }{\mathrm dt}\int_\domain
    \frac{1}{2\alpha}|\nabla\varphi|_{\ell^2}^2\,\mathrm dx +
    \int_{\bdry} \varphi\big(\vec m -
    \alpha^{-1}\partial_t\nabla\varphi\big)\cdot {\vec n}
    \,\mathrm{d}o_x = \int_\domain \vec m \cdot \nabla \varphi \,\mathrm dx.
  \end{equation*}
  Then, integrate the third equation of \eqref{eq:magnetic_euler_poisson} in
  space, use the Gauß theorem and add the result to the previous equation.
\end{proof}

From \eqref{eq:energy_balance} we note that energy conservation only holds
whenever boundary terms vanish or their effects cancel out. For instance,
this may occur whenever we consider boundary conditions $\mom \cdot \normal
= 0$ and $\nabla\varphi \cdot \normal = 0$, or periodic boundary
conditions. The energy balance for the barotropic case is briefly discussed
in Appendix \ref{AppBaro}.


\subsection{Time scales of the magnetic Euler-Poisson equations}
\label{subse:time_scales}
Special care must be taken when discretizing system
\eqref{eq:magnetic_euler_poisson} due the possible presence of highly
disparate time scales \cite{eulerpoisson21, eulerpoisson22,
eulerpoisson23}. With the goal that the time-step size shall only be
subject to a \emph{hyperbolic CFL} condition from the Euler subsystem
\cite{eulerpoisson23}, it is crucial to design the approximation technique such
that faster scales can be overstepped safely. System
\eqref{eq:magnetic_euler_poisson} contains three time scales.

\paragraph{Hyperbolic time scale $T_{\mathrm{h}}$}
The hydrodynamical subsystem is obtained by removing $\varphi$ and
$\vec\Omega$ from \eqref{eq:magnetic_euler_poisson} resulting in the
compressible Euler equations without a force. This system is---due to the
lack of a viscosity term--- \emph{scale free}. Nevertheless, heuristically,
we wish to define a hyperbolic time scale $T_{\mathrm{h}}$ as the time scale
on which the hyperbolic state $\state$ undergoes significant changes.
Concretely, for the fully discrete approximation scheme we can identify
such a time scale $T_{\mathrm{h}}$ by identifying it with the maximal
admissible time-step size $\tau_n$ given by a hyperbolic CFL condition for
the $n$th time step; see Equation~\eqref{CflFormula}:
\begin{align*}
  T_{\mathrm{h}}\,:=\,\tau_n.
\end{align*}
With this definition, we will automatically fully resolve the hyperbolic time
scale in our numerical computations.

\paragraph{Plasma oscillation time scale
$T_{\mathrm{p}}=\omega_{\mathrm{p}}^{-1}$}
Another frequency regime present in the \eqref{eq:magnetic_euler_poisson}
is the one of plasma oscillations \cite{Bitten2004, Goed2004}. We can
uncover this regime by setting $\vec\Omega$ and the hyperbolic flux to zero
in \eqref{eq:magnetic_euler_poisson} which yields a limiting equation
coupling potential $\varphi$ and momentum $\vec m$ \cite{eulerpoisson23}:
\begin{align*}
  \begin{cases}
    \begin{aligned}
      \partial_t \vec m &= -\rho\nabla\varphi,
      \\[0.5em]
      \partial_t(-\Delta\varphi) &= - \alpha\,\nabla \cdot \vec m.
    \end{aligned}
  \end{cases}
\end{align*}
Now, taking the divergence of the first equation and the time
derivative of the second, and assuming that the gradient in the density
$\rho$ is sufficiently small, we arrive at
\begin{align*}
  \Delta\big(\partial_{tt}\varphi + \omega_{\mathrm{p}}^2\varphi\big)
  = 0,
  \qquad
  \omega_{\mathrm{p}} \,:=\, \sqrt{\rho\alpha}
  \;=\; \sqrt{\frac{\rho\,q_{\mathrm{e}}^2}{\varepsilon_0m_{\mathrm{e}}^2}}.
\end{align*}
Consequently, the potential $\varphi$ (and with it the momentum $\vec m$
and the density $\rho$) can exhibit instabilities governed by a simple
harmonic oscillator equation. We note that the plasma frequency
$\omega_{\mathrm{p}}$ can be very large. For most high-energy density
applications $\omega_{\mathrm{p}}$ typically takes values in the GHz regime;
see for instance \cite[p. 12]{Bitten2004} and \cite[p. 56]{Goed2004}.

\paragraph{Cyclotron and diocotron time scales
$T_{\mathrm{c}}=\omega_{\mathrm{c}}^{-1}$,
$T_{\mathrm{d}}=\omega_{\mathrm{d}}^{-1}$}
Another set of critical time scales are related to the presence of the magnetic
field $\vec B$. Introducing the \emph{cyclotron frequency}
$\omega_{\mathrm{c}}:=|\vec\Omega|_{\ell^2}=\tfrac{q_{\mathrm e}}{m_{\mathrm
e}}|\vec B|_{\ell^2}$ we can identify another subsystem of
\eqref{eq:magnetic_euler_poisson} by formally setting the potential $\varphi$
and the hyperbolic flux to zero, viz.,
\begin{align*}
  \partial_t\vec m = \vec m \times \vec\Omega.
\end{align*}
This limiting equation is equivalent to the equation of motion of isolated
charged particles subject to a static magnetic field. We note that, in the
context of cold low-density plasma physics applications, times scales usually
follow the order $T_{h}^{-1}\ll\omega_{\mathrm{p}}\ll\omega_{\mathrm{c}}$
see \cite{eulerpoisson22, Bitten2004}. For instance, the cyclotron
frequency for isolated electrons in a magnetic field of around 1 Tesla,  is
approximately 180 GHz \cite{Bitten2004,Goed2004}. Contrary to the
plasma frequency, the cyclotron frequency does not depend on the
density.

We note that, crucially, while the cyclotron frequency $\omega_{\mathrm{c}}$
might become very large, the cyclotron motions might not necessarily manifest
in the macroscopic dynamics of the fluid. In particular, this is true for the
so-called \emph{magnetic-drift} or \emph{guiding center} limit
\cite{Degond2009, Bostan2007, eulerpoisson22}. We derive the drift limit by
assuming that inertial terms, temperature, and pressure are negligible,
meaning: $\partial_t \mom + \diver(\rho^{-1}\mom \mom^\transp) \approx \bzero$
and $\nabla p \approx \bzero$ in \eqref{eq:magnetic_euler_poisson}, which in
turn implies $-\rho\nabla\varphi + \vec m \times \vec\Omega = \vec 0$. Under
the transverse magnetic assumption, $\nabla \varphi \perp \boldsymbol \Omega$,
we solve this algebraic condition by introducing a \emph{drift velocity}
\cite[Chapter 3]{Chen1984} that depends on the prescribed magnetic field and
the potential:
\begin{align}
  \label{eq:magnetic_drift_velocity}
  \vec v_{\mathrm{dr}}\,:=\,-\frac{\nabla\varphi\times\vec\Omega}
  {|\vec\Omega|^2_{\ell^2}}.
\end{align}
System \eqref{eq:magnetic_euler_poisson} then reduces to a coupled system
of density $\rho$ and potential $\varphi$:
\begin{align}
  \label{eq:magnetic_drift_limit}
  \begin{cases}
    \begin{aligned}
      \partial_t \rho + \nabla \cdot (\rho\,\vec v_{\mathrm{dr}}) &= 0,
      \\[0.5em]
      \partial_t(-\Delta\varphi) &= - \alpha\,\nabla \cdot \vec (\rho\,\vec
      v_{\mathrm{dr}}).
    \end{aligned}
  \end{cases}
\end{align}
Substituting \eqref{eq:magnetic_drift_velocity} into the second equation of
\eqref{eq:magnetic_drift_limit}, and performing a dimensional analysis, we can
identify a time scale for the dynamics. The resulting time scale is determined
by the \emph{diocotron frequency}:
\begin{align*}
  T_{\mathrm{d}}=\omega_{\mathrm{d}}^{-1},\qquad
  \omega_{\mathrm{d}}\,:=\,\frac{\rho\,\alpha}{|\vec\Omega|_{\ell^2}}
  \,=\,\frac{\omega_{\mathrm{p}}^2}{\omega_{\mathrm{c}}}.
\end{align*}
The magnetic drift-limit model \eqref{eq:magnetic_drift_limit} is, in itself,
relatively benign as it is dominated by the slow time-scale $T_{\mathrm{d}}$ and
does not contain the fast plasma and cyclotron frequencies. However, designing
schemes capable of solving the full magnetic Euler-Poisson system
\eqref{eq:magnetic_euler_poisson} that contains the fast scales in a fluid
regime where inertial and thermal effects are negligible is non-trivial
endeavour \cite{Degond2009, Crockatt2021, Crockat2023}. In Section
\ref{sec:numerical} we will present numerical simulations where the time scales
$1\sim T_{h}^{-1}\sim\omega_{\mathrm{d}}\ll\omega_{\mathrm{p}}
\ll\omega_{\mathrm{c}}\sim10^{12}$ span 12 orders of magnitude.


\section{Operator splitting and PDE Schur complement}
\label{sec:time-stepping}%

We now discuss our strategy for approximating solutions to
\eqref{eq:magnetic_euler_poisson} that maintain a discrete counterpart of the
structural properties outlined in
Lemmas~\ref{lem:entropies_are_orthogonal}, \ref{lem:gauss_law},
\ref{lem:energy_balance}, and Corollary~\ref{cor:invariant_domain}.
Following the general approach outlined in \cite{eulerpoisson23} we first
perform an operator split of \eqref{eq:magnetic_euler_poisson} into a
hyperbolic operator and a \emph{source update}, and then condense the
resulting system on semi-discrete level with a PDE Schur complement
\cite{eulerpoisson23}.

\subsection{Operator splitting}
The discussion outlined in Section~\ref{sec:preliminaries} suggests to
treat the hydrodynamical subsystem of \eqref{eq:magnetic_euler_poisson}
separately from the update induced by the Lorentz force. We thus proceed to
approximate solutions to \eqref{eq:magnetic_euler_poisson} by an operator
splitting approach. Concretely, we intend to use a second-order
\emph{Strang} split (see \cite{Strang1968}; see also
\cite[Ch.~5]{Quarteroni1994} for an overview of operator-splitting
methods), or some form of high-order IMEX split \cite{Ern_Guermond_2022,
Ern_Guermond_2023}.

The \emph{hyperbolic subsystem} is recovered from
\eqref{eq:magnetic_euler_poisson} by taking the formal limit
$\varepsilon:=q_{\mathrm{e}}/m_{\mathrm{e}}\to0$:
\begin{align}
  \label{eq:hyperbolic_operator}
  \begin{cases}
    \begin{aligned}
      \partial_t \rho + \nabla \cdot \vec m
      &= 0,
      \\[0.5em]
      \partial_t \vec m + \nabla \cdot (\rho^{-1}\vec
      m\vec m^{\transp} + \mathbb I p)
      &= \vec 0,
      \\[0.5em]
      \partial_t \mathcal E + \nabla \cdot \big(\rho^{-1}\vec m(\mathcal E
      + p)\big) & = 0. \\
     \partial_t \Delta\Epot & = 0
    \end{aligned}
  \end{cases}
\end{align}
Conversely, the \emph{source update subsystem} is recovered from
\eqref{eq:magnetic_euler_poisson} by taking the opposite limit,
$\varepsilon:=(q_{\mathrm{e}}/m_{\mathrm{e}})^{-1}\to0$ and renormalizing the
timescale with $t=\mathcal{O}(\varepsilon)$:
\begin{align}
  \label{eq:source_update_operator}
  \begin{cases}
    \begin{aligned}
     \partial_t \rho & = 0 \\
      \partial_t \vec m &= -\rho\nabla\varphi + \vec m \times \vec\Omega,
      \\[0.5em]
      \partial_t \mathcal E & = -\vec m \cdot \nabla\varphi,
      \\[0.5em]
      \partial_t(-\Delta\varphi) &= - \alpha\,\nabla \cdot \vec m.
    \end{aligned}
  \end{cases}
\end{align}
With both subsystems at hand Strang splitting takes the following form:
\begin{definition}[Strang splitting]
  \label{def:strang_splitting}
  Given a state $\vec u^n =\big[\rho^n,\mom^n,\totme^n\big]$ and potential
  $\Epot^n$ at time $t_n$, compute an update as follows:
  \begin{itemize}
    \item[--] \textit{First fractional step (from $t_n$ to
      $t_{n+1/2}=t_n+\tau_n/2$).}  Solve the hyperbolic system
      \eqref{eq:hyperbolic_operator} with extended  initial state
      $\big\{\vec u^n, \varphi^n\big\}$ resulting in a time-step  size
      $\tau_n/2$ and an updated state $\big\{\hat{\vec u},
      \hat\varphi\big\}$ corresponding to pseudo-time $t_{n+1/2}$.
    \item[--] \textit{Second fractional step (from $t_n$ to
      $t_{n+1}=t_n+\tau_n$).} Then, solve the source update system
      \eqref{eq:source_update_operator} with initial state $\big\{\hat{\vec
      u}, \hat\varphi\big\}$ and fixed time-step size $\tau_n$ resulting in
      an updated state $\big\{\breve{\vec u}, \breve\varphi\big\}$.
    \item[--] \textit{Third fractional step (from $t_{n+1/2}$ to
      $t_{n+1}$).} We solve the hyperbolic system
      \eqref{eq:source_update_operator} a second time with initial state
      $\big\{\breve{\vec u}, \breve\varphi\big\}$ and fixed time-step size
      $\tau_n/2$ resulting in a final update $\big\{\vec u^{n+1},
      \varphi^{n+1}\big\}$ for time $t_{n+1}$.
  \end{itemize}
\end{definition}
We note here that the first fractional step, invoking the Euler solver,
determines the largest permissible $\tau_n$ compatible with the CFL
condition. Then, the second and third fractional steps use the same value
of $\tau_n$. We refer the reader to \cite{GuermondNavier2021,
eulerpoisson23, Dao2024} for a detailed discussion on the use of Strang
splitting for related hyperbolic systems.

The methodology developed in this paper can in principle be used with a
variety of different numerical discretization techniques for the Euler
subsystem \eqref{eq:hyperbolic_operator}. In the following we simply make
the assumption that the choice of discretization and solver for
\eqref{eq:hyperbolic_operator} maintains conservation, admissibility and a
suitable invariant domain on the fully discrete level; see
Section~\ref{sec:assumpHyp} and Appendix~\ref{sec:HypSolver}. We focus in
detail on the discretization and structural properties of
\eqref{eq:source_update_operator}.

\subsection{A PDE Schur complement}\label{sec:PDESchur}
We now discuss how to discretize \eqref{eq:source_update_operator} in time.
We first observe that, since $\partial_t\rho=0$, system
\eqref{eq:source_update_operator} essentially reduces to a coupled system
for the velocity $\vec v$ and the potential $\varphi$,
\begin{align}
  \label{eq:source_update_subsystem}
  \begin{cases}
    \begin{aligned}
      \partial_t \vec v &= -\nabla\varphi + \vec v \times \vec\Omega,
      \\[0.5em]
      \partial_t(-\Delta\varphi) &= - \alpha\,\nabla \cdot (\rho\vec v).
    \end{aligned}
  \end{cases}
\end{align}
The rate of change of the total mechanical energy $\mathcal{E}$ is then
determined by
\begin{align}
  \label{eq:source_update_energy}
  \partial_t\mathcal{E}=-\rho\vec v\cdot\nabla\varphi
\end{align}
and simply equates to the rate of change of the kinetic energy. In order to
accommodate different IMEX time-stepping techniques, we now discretize
\eqref{eq:source_update_subsystem} in time with a $\theta$-scheme, where
$\theta=1$ recovers a backward Euler step and $\theta=1/2$ is the
Crank-Nicolson scheme.
\begin{align}
  \label{eq:semi_discrete_theta_scheme}
  \begin{cases}
    \begin{aligned}
      \vec v^{n+\theta} &= \vec v^n - \theta\,\tau_n\nabla \varphi^{n+\theta}
      + \theta\,\tau_n \vec v^{n+\theta} \times \vec \Omega,
      \\[0.5em]
      -\Delta\varphi^{n+\theta} & = -\Delta\varphi^n -
      \theta\,\tau_n\alpha\,\nabla\cdot(\rho^n\vec v^{n+\theta}),
      \\[0.5em]
      \vec v^{n+1} &= \frac{1}{\theta}\,\vec v^{n+\theta} -
      \frac{1-\theta}{\theta}\,\vec v^{n},
      \\[0.5em]
      \varphi^{n+1} &= \frac{1}{\theta}\,\varphi^{n+\theta} -
      \frac{1-\theta}{\theta}\,\varphi^{n}.
    \end{aligned}
  \end{cases}
\end{align}
Note that the first two lines of \eqref{eq:semi_discrete_theta_scheme}
correspond with a coupled system, while the last two lines are just an
extrapolation step.

\begin{lemma}[Semi-discrete parabolic energy balance]
  \label{lem:semi_discrete_parabolic_energy_balance}%
  The semi-discrete system \eqref{eq:semi_discrete_theta_scheme} admits an
  energy balance:
  \begin{multline}
    \frac1{2\alpha}\int_\domain\alpha\rho^n|\vec v^{n+1}|_{\ell^2}^2
    +|\nabla\varphi^{n+1}|_{\ell^2}^2\,\mathrm{d}x
    \\
    +\big(\theta-\frac12\big)\,
    \frac1{2\alpha}\int_\domain\alpha\rho^n|\vec v^{n+1} - \vec
v^{n}|_{\ell^2}^2
    +|\nabla\varphi^{n+1} - \nabla\varphi^{n}|_{\ell^2}^2\,\mathrm{d}x
    \\
    = \frac1{2\alpha}\int_\domain\alpha\rho^n|\vec v^{n}|_{\ell^2}^2
    +|\nabla\varphi^{n}|_{\ell^2}^2\,\mathrm{d}x
    \\
    +\tau_n\int_{\bdry}
    \Big\{\frac{1}{\alpha}\frac{\nabla\varphi^{n+\theta} -
    \nabla\varphi^{n}}{\theta\tau_n}
    \,-\,\rho^n\vec v^{n+\theta}\Big\}\cdot\vec
    n_{\bdry}\,\varphi^{n+\theta}\,\mathrm{d}o_x.
  \end{multline}
\end{lemma}
\begin{proof}
  Test the first equation of \eqref{eq:semi_discrete_theta_scheme} with
  $\theta^{-1}\varphi^n\vec v^{n+\theta}$ and the second equation with
  $(\alpha\theta)^{-1}\varphi^{n+\theta}$. Then, the identity follows by
  integration by parts and observing that the following algebraic
  \emph{polarization} identity holds true:
  \begin{align*}
    \frac{1}{\theta}\big(\theta a+(1-\theta)b\,,\,\theta
    a+(1-\theta)b-b\big)\,=\,\frac{1}{2}\big(|a|^2-|b|^2\big)
    +\big(\theta-\frac{1}{2}\big)|a-b|^2.
  \end{align*}
\end{proof}
Following the general principle discussed in
\cite[Sec.~3.2]{eulerpoisson23}, we now transform the coupled system
\eqref{eq:semi_discrete_theta_scheme} into a lower triangular form that
allows us to first solve for the new potential $\varphi^{n+\theta}$ and
then compute the new velocity field $\vec v^{n+\theta}$:
\begin{align}
  \label{eq:semi_discrete_theta_scheme_condensed}
  \begin{cases}
    \begin{aligned}
      -\Delta\varphi^{n+\theta} - \theta^2\tau_n^2\alpha\, \nabla \cdot
      \left(\rho^n\mathcal B_{n,\theta}^{-1}\nabla\varphi^{n+\theta}\right)
      &=
      -\Delta\varphi^n - \theta\tau_n\alpha\,\nabla\cdot\left(\rho^n\mathcal
      B_{n,\theta}^{-1} \vec v^n\right),
      \\[0.5em]
      \vec v^{n+\theta} & = \mathcal B_{n,\theta}^{-1}\left(\vec v^n -
      \theta\tau_n\,\nabla \varphi^{n+\theta}\right),
    \end{aligned}
  \end{cases}
\end{align}
where $\mathcal B_{n,\theta}:\mathbb{R}^d \rightarrow \mathbb{R}^d$ is
defined by
\begin{align}
  \label{eq:definition_bntheta}
  \mathcal B_{n,\theta} \vec v
  \,:=\,
  \vec v - \theta\,\tau_n\,\vec v \times \vec \Omega \ \ \text{for } \vec v\in
\mathbb{R}^d.
\end{align}
Note that $\det(\mathcal{B}_{n,\theta}) =
1+\theta^2\tau_n^2|\vec\Omega|_{\ell^2}^2 =
1+\theta^2\tau_n^2\omega_{\mathrm{c}}^2$,
moreover, the inverse can be written as
\begin{align}
  \label{eq:B_inverse}
  \mathcal{B}_{n,\theta}^{-1}\vec v = \frac{\vec v + \theta\tau_n\vec
  v\times\vec\Omega + \theta^2\tau_n^2(\vec v\cdot\vec\Omega)\vec \Omega}
  {1+\theta^2\tau_n^2|\vec \Omega|^2_{\ell}}.
\end{align}
We note in passing that formulation
\eqref{eq:semi_discrete_theta_scheme_condensed}-\eqref{eq:definition_bntheta}
is closely related to what is usually known as static condensation in the
finite element literature \cite{eulerpoisson22}. Here, we set $\vec
v\cdot\vec\Omega=0$ in the 2D case ($d=2$), consistent with our convention of
using a transverse magnetic configuration. We make the following observation
for the bilinear form arising from the weak formulation of
\eqref{eq:semi_discrete_theta_scheme_condensed}.
\begin{lemma}[Well posedness]\label{EllipticWellPosed}
  Suppose that $\rho^n\in L^\infty(\domain)$. Then, the bilinear form
  \begin{align*}
    a_{n,\theta}(\varphi,\psi)\,:=\,
    \big(\nabla\varphi,\nabla\psi)_{\Ltwod}+ \theta^2\tau_n^2\alpha
    \big(\rho^n\mathcal{B}_{n,\theta}^{-1}\nabla\varphi,
    \nabla\psi)_{\Ltwod}
  \end{align*}
  is coercive and bounded on $H_0^1(\domain)$; and the
  boundedness constant  $C$ formally scales with
  $C=\mathcal{O}(\tau_n^2\omega_{\mathrm{p}}^2)$.
\end{lemma}
\begin{remark}
  The previous lemma also holds true when the space $H_0^1(\domain)$ is
  replaced by $H^1(\domain)/\mathbb{R}$, the space of $H^1$ functions with
  zero mean value, which is appropriate for homogeneous Neumann boundary
  conditions. Another possible choice is
  $H_{\mathrm{per}}^1(\Omega)/\mathbb{R}$, the space of $H^1$ functions
  with periodic boundary conditions. In the reminder of the paper, we use
  homogeneous Dirichlet boundary conditions to avoid digression into
  technicalities that are not central to the main ideas advanced in the
  paper. The results can be easily adapted to other boundary conditions of
  technical interest.
\end{remark}

It is of interest to discuss briefly in what form subsystem
\eqref{eq:semi_discrete_theta_scheme} can be made compatible with the
magnetic drift limit \eqref{eq:magnetic_drift_limit}. Notably, mathematical
results related to the convergence of limiting processes, e.g., by
weak-defect measure limits \cite{Golse1999}, raise the uncomfortable
question whether such an endeavour is actually numerically meaningful.
Nevertheless, we report the following result:
\begin{lemma}[Compatibility with magnetic drift limit]
  \label{lem:semi_discrete_parabolic_coercive}
  Suppose that at $t_n$ the state $[\rho^n,\vec v^n,\varphi^n]$ satisfies
  the following algebraic conditions:
  \begin{align*}
    \nabla\varphi^n\perp\vec\Omega,\qquad
    \vec v^n =\,
    -\frac{\nabla\varphi^n\times\vec\Omega}{|\vec\Omega|^2_{\ell^2}},
    \qquad
    \nabla\cdot\big(\rho^n\vec v^n\big) = 0.
  \end{align*}
  Then, the solution to \eqref{eq:semi_discrete_theta_scheme} satisfies
  $\vec v^{n+\theta}=\vec v^n$ and $\varphi^{n+\theta}=\varphi^{n}$,
  i.\,e., the algebraic conditions are preserved.
\end{lemma}

\begin{proof}
  We start by rewriting the first equation of
  \eqref{eq:semi_discrete_theta_scheme_condensed} by subtracting the same terms
  for $\varphi^n$ that appear on the left-hand side of the equation. Using the
  first and second assumption to manipulate the right hand side shows:
  \begin{align*}
    -\Delta\big\{\varphi^{n+\theta}-\varphi^{n}\big\}
    -\theta^2\tau_n^2\alpha\,&\nabla\cdot
    \Big(\rho^n\mathcal
    B_{n,\theta}^{-1}\nabla\big\{\varphi^{n+\theta}-\varphi^n\big\}\Big)
    \\
    &= - \theta\tau_n\alpha\,\nabla\cdot\Big(\rho^n\mathcal
    B_{n,\theta}^{-1} \vec v^n\Big) +\theta^2\tau_n^2\alpha\,\nabla\cdot
    \Big(\rho^n\mathcal B_{n,\theta}^{-1}\nabla\varphi^n\Big)
    \\
    &= - \theta\tau_n\alpha\,\nabla\cdot\Big(\rho^n\mathcal
    B_{n,\theta}^{-1}\Big\{\vec v^n +
    \theta\tau_n\,
    \frac{\nabla\varphi^n\times\vec\Omega}
    {|\vec\Omega|^2_{\ell^2}}\times\vec\Omega\Big\}\Big).
    \\
    &= - \theta\tau_n\alpha\,\nabla\cdot\big(\rho^n\mathcal
    B_{n,\theta}^{-1}
    \mathcal{B}_{n,\theta}\vec v^n\big)\;=\;0.
  \end{align*}
  the left-hand side is a coercive operator under the assumption of
  homogeneous Dirichlet boundary data, thus
  $\varphi^{n+\theta}=\varphi^{n}$. Next, using the second algebraic
  condition, the second equation of
  \eqref{eq:semi_discrete_theta_scheme_condensed} can be rewritten as
  follows:
  \begin{align*}
    \mathcal{B}_{n,\theta}\,\big\{\vec v^{n+\theta}-\vec v^n\big\}
    = \theta\tau_n\big(\nabla\varphi^{n+\theta}-\nabla\varphi^{n}\big) = 0,
  \end{align*}
  implying that $\vec v^{n+\theta}=\vec v^n$.
\end{proof}


\begin{remark}[Asymptotic compatibility with equilibrium solutions]
  \label{rem:stationary_solution}
  As an immediate consequence, we observe that any flow configuration with a
  barotropic equation of state in the regime of \emph{negligible pressure}
  and with a rotationally symmetric density profile (that only depends on the
  radius $r$ in cylindrical coordinates) is a (almost) stationary solution to
  the Euler-Poisson system (on the continuous and semidiscrete level):
  \begin{align*}
    \rho\,=\,\rho(r),
    \qquad
    p=p(\rho)\approx 0,
    \qquad
    \vec v^n=\,-\frac{\nabla\varphi^n\times\vec\Omega}{|\vec\Omega|^2_{\ell^2}}.
  \end{align*}
\end{remark}


\section{A fully discrete structure-preserving discretization}
\label{sec:FullyDiscrete}
We now describe a spatial discretization of
\eqref{eq:semi_discrete_theta_scheme_condensed} that maintains the crucial
equivalence of \eqref{eq:semi_discrete_theta_scheme} and
\eqref{eq:semi_discrete_theta_scheme_condensed} on a fully discrete level.
This in turn enables us to establish discrete counterparts of
Lemmas~\ref{lem:semi_discrete_parabolic_energy_balance}
and~\ref{lem:semi_discrete_parabolic_coercive}. We start by introducing the
finite element spaces and a lumped inner product. Then, we discuss the
fully discrete system.

\subsection{Finite element spaces}\label{sec:finite_elements}
Let $\mathbb Q_1$ denote the space of $d$-variable polynomials of degree at
most $1$ in each variable, and let $\widehat K = [0,1]^d$ denote the
reference cell. Let $\mathcal T_h$ be a \emph{shape and form regular} mesh
consisting of quadrilateral (for $d=2$) or hexahedral (for $d=3$) cells
$K$, covering the domain $\domain$, and obtained from a diffeomorphic
mapping $\vec T_K : \widehat K \to K$. We then introduce a discontinuous
and a continuous finite element space as follows:
\begin{align}
  \nonumber
  \mathbb V_h &:= \{ z_h \in \Ltwo \;:\; z_h \circ \vec T_K \in \mathbb
  Q_1 \text{ for all } K \in \mathcal T_h \},
  \\[0.5em]
  \label{HypSpace}
  \mathbb H_h &:= \{ \psi_h \in \mathcal C^0(\domain) \;:\; \psi_h \circ
  \vec T_K \in \mathbb Q_1 \text{ for all } K \in \mathcal T_h \}.
\end{align}
Note that $\text{dim}\{\mathbb H_h\} \leq \text{dim}\{\mathbb V_h\}$ and
$\mathbb H_h \subseteq \mathbb V_h$. That is, by construction every
function in $u_h \in \mathbb H_h$ automatically belongs to $\mathbb V_h$.
We will use $\{\phi_{i,K}\}$ to denote the nodal basis of the discontinuous
space $\mathbb V_h$, and $\{\chi_i\}$ for the nodal basis of the continuous
space $\mathbb H_h$. As it is standard for $\mathbb Q_1$ elements, we will
assume that the interpolation points are associated with the vertices of
the elements. See \cite{ErnGuermond2004, DiPietro2012} for more details on
the construction of nodal continuous and discontinuous finite element
spaces.

The mapping $\vec T_K : \widehat K \to K$ is not assumed to be affine.
Therefore, it is \emph{not} necessarily true that $\nabla \mathbb H_h
\subset \mathbb V_h^d$, which, on the semidiscrete level, is a crucial property
to establish equivalence between
\eqref{eq:semi_discrete_theta_scheme} and
\eqref{eq:semi_discrete_theta_scheme_condensed}. Furthermore, in order to
connect a fully discrete energy balance to the energy balance of the
(algebraic) approximation of the hyperbolic subsystem, we need to extract
the kinetic energy as an algebraic sum over collocated point values in
support points rather than an integral \cite{eulerpoisson23}. For this
reason we need to make use of a carefully constructed lumped inner product
with corresponding nodal interpolant. We define
\begin{definition}[Lumped bilinear product]
  Let $\mathcal{C}(\mathcal T_h)$ be the space of piecewise continuous
  scalar-valued functions on the mesh $\mathcal T_h$. That is, functions
  that are continuous within each element, with well defined traces within
  each element, but possibly discontinuous across element boundaries. Let
  $f(\xcoord), g(\xcoord) \in \mathcal{C}(\mathcal T_h)$, then we define
  the bilinear form:
  \begin{align}
    \label{eq:lumped_inner_product}
    \langle f, g \rangle_h \,:=\,
    \sum_{j=1}^{\mathcal{N}} \sum_{K \in \mathcal T_h} m_{j,K}
    \, f(\vec x_{j})\big|_K\, g(\vec x_{j})\big|_K
  \end{align}
  where $m_{j,K}:=\int_K \phi_{j,K}(\vec x)\,\text d x$ is an element of
  the lumped mass matrix, while $f(\vec x_{j})\big|_K$ and $g(\vec
  x_{j})\big|_K$ denote interior traces with respect to the element $K$.
  Similarly, we define the space of piecewise continuous vector-valued
  functions on the mesh as $\mathcal{C}(\mathcal T_h)^d$. Let
  $\boldsymbol{f}(\xcoord), \boldsymbol{g}(\xcoord) \in
  \mathcal{C}(\mathcal T_h)^d$, then we define the bilinear form:
  \begin{align}
    \label{eq:lumped_inner_product_ld}
    \langle \boldsymbol{f}, \boldsymbol{g} \rangle_h \,:=\,
    \sum_{j=1}^{\mathcal{N}} \sum_{K \in \mathcal T_h} m_{j,K}
    \,\boldsymbol{f}(\vec x_{j})\big|_K \cdot \boldsymbol{g}(\vec x_{j})\big|_K
  \end{align}
  If the functions $f(\xcoord)$ and $g(\xcoord)$ in definition
  \eqref{eq:lumped_inner_product} belong to the space $\mathbb{V}_h$ then
  we have that $\langle \cdot, \cdot\rangle_h$ defines a proper
  inner-product, also known as lumped inner-product. The corresponding
  induced norm shall be denoted as $\|f\|_h:=\sqrt{\langle f,f\rangle_h}$
  and the following norm equivalence holds:
  \begin{align}
    \label{equivBounds}
    c_1 \|f\|_{L^2(\domain)} \leq \|f\|_h \leq c_2 \|f\|_{L^2(\domain)}
    \quad \text{for all }f \in \mathbb{V}_h
  \end{align}
  for some non-negative bounded constants $c_1$ and $c_2$. A similar
  statement holds for the case definition
  \eqref{eq:lumped_inner_product_ld} when $\boldsymbol{f}(\xcoord),
  \boldsymbol{g}(\xcoord) \in \mathbb{V}_h^d$. More broadly, the lumped
  bilinear products \eqref{eq:lumped_inner_product} and
  \eqref{eq:lumped_inner_product_ld} induces a norm in some
  finite-dimensional space provided we can establish equivalence bounds
  such as \eqref{equivBounds}. We extend the use of these bilinear forms in
  Assumption \ref{ass:regular_mesh}.
\end{definition}
\begin{remark}[Lumping accuracy]\label{lumpingAccuracy}
  We note in passing that
  \begin{align}
    \sum_{j=1}^{\mathcal{N}}
    \sum_{K \in \mathcal T_h}
    m_{j,K} \varphi_h\big|_K(\vec x_{j}) =
    \int_{\domain }
    \varphi_h\,\mathrm{d}\xcoord
    \ \ \text{for all } \varphi_h \in \mathbb{V}_h
  \end{align}
  In other words, the quadrature rule using weights $m_{j,K}$ is exact for
  functions in the space $\mathbb{V}_h$. Proving second-order accuracy of
  this lumping when using affine mappings is a relatively well-known
  exercise, see \cite{Ciarlet1978} and \cite[Ch. 3]{Bartels2015}. However,
  establishing the same property for non-affine meshes is a much more
  technical proposition \cite[Ch. 13]{ErnGuermond2021book}. For
  asymptotically affine meshes, second order consistency of the lumping is
  observed experimentally; see, for instance, the computational results in
  \cite[Section 5.1]{eulerpoisson22}.
\end{remark}

We observe that $\langle\varphi_h,\psi_h\rangle_h$ is a proper inner
product on $\mathbb V_h$ and $\mathbb H_h$ approximating the $\Ltwo$ inner
product. We extend the inner product to vector-valued functions
$\vec\varphi_h$, $\vec\psi_h\in\mathbb{V}_h^d$ by setting
\begin{equation*}
    \langle \vec\varphi_h, \vec\psi_h \rangle_h \,:=\,
    \sum_{j=1}^{\mathcal{N}} \sum_{K \in \mathcal T_h} m_{j,K}
    \,\vec\varphi_h\big|_K(\vec x_{j})\cdot \vec\psi_h\big|_K(\vec x_{j}).
\end{equation*}
The corresponding nodal interpolant $I_h\,:\mathcal C^0(\mathcal
T_h)^d\to\mathbb{V}_h^d$ is given by
\begin{equation}
  \label{eq:nodal_interpolant}
  \big(I_h\vec \varphi\big)(\vec x)\,:=\,
  \sum_{j=1}^{\mathcal{N}} \sum_{K \in \mathcal T_h}
  \vec\varphi\big|_K(\vec x_{j})\,\phi_{j,K}(\vec x).
\end{equation}
By construction, it holds true that
\begin{align*}
  \langle I_h\vec\varphi, \vec\psi \rangle_h =\langle \vec\varphi, \vec\psi
  \rangle_h =\langle \vec\varphi, I_h\vec\psi \rangle_h \quad\text{for
  all}\quad \vec\varphi, \vec\psi\in\mathcal C^0(\mathcal T_h)^d.
\end{align*}
where $C^0(\mathcal T_h)^d$ is the space of piecewise continuous function on
the triangulation $\mathcal T_h$.

\subsection{Assumptions on the hyperbolic solver}\label{sec:assumpHyp}
We now summarize our structural assumptions on the discretization of the
hyperbolic subsystem. For the sake of completeness, we outline some
implementation details of the hyperbolic solver used in our numerical
computations in Appendix \ref{sec:HypSolver}. In general, we make the
following assumptions on the hyperbolic update procedure.

We assume that all components of the hyperbolic state approximation $\vec
u_h^n=\big[\rho_h^n,\mom_h^n,\totme_h^n\big]$ for time $t_n$ are
discretized using the scalar-valued, discontinuous finite element space
$\mathbb V_h$ introduced above:
\begin{align*} \vec u_h^n(\xcoord) = \sum_{j=1}^{\mathcal{N}} \sum_{K \in
\mathcal T_h} \vec U^n_{i,K} \phi_{i,K}(\xcoord), \end{align*}
where $\vec U^n_{i,K}=\big[\rho_i^n,\mom_i^n,\totme_i^n\big]\in\mathbb
R^{d+2}$. Let $\vec u_h^{n+1}$ be the discrete update computed with the
hyperbolic solver for time $t_{n+1}$. The hyperbolic update shall maintain
an admissible set, this is to say that
\begin{align*}
  \vec U^n_{i,K}\in\mathcal{A}\text{ for all }i\le\mathcal{N},\,K
  \in\mathcal{T}_h
  \quad\text{implies}\quad
  \vec U^{n+1}_{i,K}\in\mathcal{A}\text{ for all }i\le\mathcal{N},\,K
  \in\mathcal{T}_h,
\end{align*}
where $\mathcal{A} = \{\mathbb{R}^{d+2} \, | \, \rho > 0 \text{ and }
e(\state) > 0\}$ for the case of the full Euler equations with covolume
equation of state. For the case of barotropic equation of state, when the
energy equation is not part of the system, admissibility reduces to
positivity of the density, $\mathcal{A} = \{\mathbb{R}^{d+1} \, | \, \rho >
0 \}$.

\paragraph{Conservation}
Neglecting the influence of boundary conditions, we assume that the
hyperbolic solver is conservative, i.\,e.,
\begin{align*}
  \sum_{j=1}^{\mathcal{N}} \sum_{K \in \mathcal T_h}
  m_{i,K}\,\vec U^{n+1}_{i,K}
  =
  \sum_{j=1}^{\mathcal{N}} \sum_{K \in \mathcal T_h}
  m_{i,K}\,\vec U^{n}_{i,K}.
\end{align*}

\paragraph{Entropy inequalities}
Finally, we assume that the hyperbolic solver maintains in some sense a
global entropy inequality: Selecting a suitable entropy $\eta$ and
neglecting the influence of boundary conditions,
\begin{align}
  \label{EulerSolverDissipation}
  \sum_{j=1}^{\mathcal{N}} \sum_{K \in \mathcal T_h} m_{i,K}
  \eta(\state_{i,K}^{n+1})
  \;\leq\;
  \sum_{j=1}^{\mathcal{N}} \sum_{K \in \mathcal T_h} m_{i,K}
  \eta(\state_{i,K}^{n}).
\end{align}
%


\subsection{Fully discrete source update}
We are now in a position to introduce a fully discrete counterpart to
\eqref{eq:semi_discrete_theta_scheme_condensed}. We proceed as follows.

\begin{definition}[Fully discrete source update]
  Given discrete $\rho^n_h\in\mathbb V_h$, $\vec v^n_h\in\mathbb V_h^d$,
  $\varphi^n_h\in\mathbb H_h$ at time $t_n$, compute approximations $\vec
  v^{n+1}_h\in\mathbb V_h^d$ and $\varphi^{n+1}_h\in\mathbb H_h$ at time
  $t_{n+1}=t_n+\tau_n$ as follows:
  \begin{align}
    \label{eq:fully_discrete_theta_scheme_condensed}
    \begin{cases}
      \begin{aligned}
        a^{n,\theta}_h(\varphi_h^{n+\theta}, \psi_h)
        &=
        \left(\nabla \varphi_h^n, \nabla \psi_h\right) +
        \theta\tau_n\alpha\, \big\langle\rho_h^n\mathcal
        B_{n,\theta}^{-1} \vec v_h^n, \nabla \psi_h\big\rangle_h,
        &\;& \forall\,\psi_h\in\mathbb{H}_h,
        \\[0.5em]
        \big\langle\vec v_h^{n+\theta}, \vec z_h\big\rangle_h
        &=
        \big\langle\mathcal B_{n,\theta}^{-1}\big(\vec v_h^n -
        \theta\tau_n\nabla\varphi_h^{n+\theta}\big), \vec
        z_h\big\rangle_h,
        &\;& \forall\,\vec z_h\in\mathbb{V}_h^d,
        \\[0.5em]
        \theta\,\vec v^{n+1} &= \vec v^{n+\theta} - (1-\theta)\,\vec v^{n},
        \\[0.5em]
        \theta\,\varphi^{n+1} &= \varphi^{n+\theta} - (1-\theta)\,\varphi^{n}.
      \end{aligned}
    \end{cases}
  \end{align}
  Here, we have introduced the discrete bilinear form
  \begin{align}
    \label{eq:discrete_bilinear_form}
    a^{n,\theta}_h(\varphi_h, \psi_h) := \big(\nabla \varphi_h, \nabla
    \psi_h\big) + \theta^2\tau_n^2\alpha\, \big\langle\rho_h^n\mathcal
    B_{n,\theta}^{-1} \nabla \varphi_h, \nabla \psi_h\big\rangle_h.
  \end{align}
\end{definition}

\subsection{Well-posedness and structure preservation}
We now establish that, for sufficiently regular mesh families $\{\mathcal
T_h\}_h$, the bilinear form $a^{n,\theta}_{h>0}$ is uniformly coercive and
bounded with a constant independent of $\omega_{\mathrm{c}}$ and essentially
only depending on $\tau_n^2\omega_{\mathrm{p}}^2$. To this end, we require
that the selected mesh sequence $\{\mathcal T_h\}_{h\geq0}$ is sufficiently
regular such that the following assumption holds true on the approximation
quality of the lumped inner product.
\begin{assumption}
  \label{ass:regular_mesh}
  There exist constants $c_0, c_1, c_2, s > 0$ such that, for all $h > 0$,
  \begin{align}
    \label{ass:lumping}
    &\big|\|\vec v_h\|_h - \|\vec v_h\|_{\Ltwod}\big| \le
      c_0 \,h^s\,\|\vec v_h\|_{\Ltwod}
    \quad\text{for all }\vec v_h\in\mathbb V_h^d, \\
    \label{ass:lumping_grad}
    &c_1 \|\nabla\varphi_h\|_{L^2(\domain)} \leq \|\nabla\varphi_h\|_h
      \leq c_2 \|\nabla\varphi_h\|_{L^2(\domain)}
    \ \text{for all }\varphi_h \in \mathbb H_h.
  \end{align}
  Assumption \eqref{ass:lumping} holds true provided the lumping is
  better than first-order accurate. For affine meshes this poses no concern
  \cite[Chapter 11]{ErnGuermond2021book}. For non-affine meshes, this depends
  on the quality of the mesh sequence; see Remark \ref{lumpingAccuracy} for
  more details.
\end{assumption}

\begin{proposition}[Coercivity and boundedness]
  \label{prop:boundedness}
  Suppose that Assumption~\ref{ass:regular_mesh} holds true and that the
  density remains positive. Then, there exists a constant $c$ independent
  of $h$ (and $\omega_p$), such that for all
  $\varphi_h,\psi_h\in\mathbb{V}_h^d$:
  \begin{align*}
    a_{h}^{n,\theta}(\varphi_h,\varphi_h)
    &\,\geq\,
    \|\nabla \varphi_h\|_{\Ltwod}^2,
    \\[0.5em]
    \big|a_{h}^{n,\theta}(\varphi_h,\psi_h)\big|
    &\,\leq\,
    \big(1 + c\,\tau_{n}^2 \alpha \|\rho_h^n\|_{h,\infty}\big)
    \|\nabla\varphi_h\|_{\Ltwod}
    \|\nabla\psi_h\|_{\Ltwod}.
  \end{align*}
  Here, $\|\rho_h^n\|_{h,\infty}$ denotes the largest value of
  $|\rho_h^n|$ obtained in the support points of $\mathbb{V}_h$.
\end{proposition}

\begin{proof}
  The first inequality follows directly from the definition of
  $a^{n,\theta}_h$ and the inverse \eqref{eq:B_inverse},
  \begin{multline*}
    a^{n,\theta}_h(\varphi_h,\varphi_h) =
    \|\nabla\varphi_h\|_{\Ltwod}^2
    \\[0.25em]
    +\;
    \frac{\theta^2\tau_n^2\alpha}{(1 + \theta^2\tau_n^2|\vec
    \Omega|_{\ell^2}^2)}
    \Big\{
      \left\langle \rho_h^n \nabla\varphi_h , \nabla\varphi_h \right\rangle_h
      +
      \tau_n^2
      \left\langle \rho_h^n \nabla\varphi_h \cdot \vec\Omega ,
      \nabla\varphi_h \cdot \vec\Omega \right\rangle_h
    \Big\}
    \\
    \geq \|\nabla\varphi_h\|_{\Ltwod}^2.
  \end{multline*}
  where we have use the assumption $\rho_h^n(\xcoord) > 0$. This
  establishes the first inequality. For the second inequality we estimate
  similarly:
  \begin{multline}
      \big|a^{n,\theta}_h(\varphi_h,\psi_h)\big|
      \leq
      \|\nabla\varphi_h\|_{\Ltwod}
      \|\nabla\psi_h\|_{\Ltwod}
      \\[0.25em]
      +\;
      \theta^2\tau_n^2\alpha\|\rho_h^n\|_{h,\infty}
      \,
      \frac {1 + 2 \tau_n |\vec\Omega|_{\ell^2} + \tau_n^2
      |\vec\Omega|_{\ell^2}^2} {1+\tau_n^2|\vec\Omega|_{\ell^2}^2}
      \,
      \|\nabla\varphi_h\|_h
      \|\nabla\psi_h\|_h,
  \end{multline}
  which combined with estimate \eqref{ass:lumping_grad} establishes the
  result.
\end{proof}
We note
\begin{corollary}
  Since the semi-norm $\|\nabla \varphi_h\|_{L^2(\domain)}$ and the the norm
  $\|\varphi_h\|_{H^1(\domain)}$ are equivalent in the $H_0^1(\domain)$ space,
  the fully discrete update \eqref{eq:fully_discrete_theta_scheme_condensed} is
  well posed.
\end{corollary}
Now, we establish the main result of this section.
\begin{proposition}[Fully discrete parabolic energy balance]
  \label{prop:fully_discrete_parabolic_energy_balance}%
  Writing
  \begin{align*}
    \rho^\nu_h\,=:\,\sum_{j=1}^{\mathcal{N}}\sum_{K\in\mathcal{T}_h}
    \rho_{j,K}^\nu\phi_{j,K},
    \qquad
    \vec v^\nu_h\,=:\,\sum_{j=1}^{\mathcal{N}}\sum_{K\in\mathcal{T}_h}
    \vec v_{j,K}^\nu\phi_{j,K},
  \end{align*}
  for $\nu=n,n+\theta$, the fully discrete system
  \eqref{eq:fully_discrete_theta_scheme_condensed} admits an energy
  balance:
  \begin{multline}
    \label{eq:fully_discrete_parabolic_energy_balance}
    \sum_{j=1}^{\mathcal{N}}\sum_{K\in\mathcal{T}_h}\frac{1}{2}m_{j,K}
    \rho_{j,K}^{n}\big|\vec v_{j,K}^{n+1}\big|_{\ell^2}^2
    \,+\,
    \frac1{2\alpha}\int_\domain\big|\nabla\varphi_h^{n+1}
    \big|_{\ell^2}^2\,\mathrm{d}x
    \\
    +\;\big(\theta-\frac12\big) \bigg\{
    \sum_{j=1}^{\mathcal{N}}\sum_{K\in\mathcal{T}_h}\frac{1}{2}m_{j,K}
    \rho_{j,K}^{n}\big|\vec v_{j,K}^{n+1}-\vec v_{j,K}^{n}\big|_{\ell^2}^2
    \,+\,
    \frac1{2\alpha}\int_\domain\big|\nabla\varphi_h^{n+1}-\nabla\varphi_h^{n}
    \big|_{\ell^2}^2\,\mathrm{d}x \bigg\}
    \\
    =\;\sum_{j=1}^{\mathcal{N}}\sum_{K\in\mathcal{T}_h}\frac{1}{2}m_{j,K}
    \rho_{j,K}^{n}\big|\vec v_{j,K}^{n}\big|_{\ell^2}^2
    \,+\, \frac1{2\alpha}\int_\domain\big|\nabla\varphi_h^{n}
    \big|_{\ell^2}^2\,\mathrm{d}x.
  \end{multline}
\end{proposition}

\begin{proof}
  We first observe that the following identity holds for any
  $\vec v_h, \vec z_h \in \mathcal C^0(\mathcal T_h)^d$ and any
  $\mathcal A \in \mathbb R^{d\times d}$:
  \begin{equation*}
    \langle \mathcal A \vec v_h, \vec z_h \rangle_h
    =
    \langle \vec v_h, \mathcal A^\mathsf T \vec z_h \rangle_h,
  \end{equation*}
  where $\mathcal A^\mathsf T$ is the transpose of $\mathcal A$. Thus, by
  replacing $\vec z_h$ with $I_h(\mathcal B_{n,\theta}^\mathsf T\vec z_h)$ in
  the second equation of \eqref{eq:fully_discrete_theta_scheme_condensed} and
  using \eqref{eq:definition_bntheta}, we obtain the identity
  \begin{equation*}
    \left\langle \vec v_h^{n+\theta}, \vec z_h \right\rangle_h
    =
    \left\langle
      \vec v_h^n,
      \vec z_h
    \right\rangle_h
      - \theta \tau_n
    \left\langle
      \nabla \varphi_h^{n+\theta},
      \vec z_h
    \right\rangle_h
    +
    \theta \tau_n
    \left\langle
      \vec v_h^{n+\theta} \times \vec \Omega,
      \vec z_h
    \right\rangle_h.
  \end{equation*}
  Next, due to the use of the lumped inner product, the second equation of
  \eqref{eq:fully_discrete_theta_scheme_condensed} also holds true when
  replacing $\boldsymbol z_h$ with $I_h(\rho_h^n\boldsymbol z_h)$:
  \begin{align*}
    \big\langle\vec v_h^{n+\theta}, \rho_h^n\vec z_h\big\rangle_h
    =
    \big\langle\mathcal B_n^{-1}\big(\vec v_h^n -
    \theta\tau_n\nabla\varphi_h^{n+\theta}\big), \rho_h^n\vec z_h
    \big\rangle_h.
  \end{align*}
  Rearranging the first equation of
  \eqref{eq:fully_discrete_theta_scheme_condensed} and substituting the
  above identities recovers a symmetric version of the update:
  \begin{equation}
    \label{eq:fully_discrete_theta_scheme_symmetric}
    \begin{cases}
      \begin{aligned}
        \left(\nabla \varphi_h^{n+\theta}, \nabla \psi_h\right)
        &=
        \left(\nabla \varphi_h^n, \nabla \psi_h\right)
        +
        \theta\tau_n\alpha\,
        \big\langle\rho_h^n \vec v_h^{n+\theta}, \nabla \psi_h \big\rangle_h,
        &\;& \forall\,\psi_h\in\mathbb{H}_h,
        \\[0.5em]
        \left\langle \vec v_h^{n+\theta}, \vec z_h \right\rangle_h
        &=
        \left\langle
          \vec v_h^n - \theta \tau_n \nabla \varphi_h^{n+\theta}
          +
          \theta \tau_n \vec v_h^{n+\theta} \times \vec \Omega,
          \vec z_h
        \right\rangle_h,
        &\;& \forall\,\vec z_h\in\mathbb{V}_h^d.
      \end{aligned}
    \end{cases}
  \end{equation}
  The result now follows similarly to the proof of
  Lemma~\ref{lem:semi_discrete_parabolic_energy_balance} by testing with
  \begin{align*}
    \psi_h:=(\alpha\theta)^{-1}\varphi_h^{n+\theta},
    \quad\text{and}\quad\vec z_h:=\theta^{-1}I_h(\rho_h^{n}\vec v^{n+\theta}_h).
  \end{align*}
\end{proof}
We are now in a position to formulate the full source update procedure.
\begin{definition}[Full source update procedure on conserved states]
  \label{defi:source_update}
  Given a conserved state $\big[\rho^n_h,\vec m^n_h, \mathcal
  E^n_h]^T\in\mathbb V_h^{d+2}$ and $\varphi^n_h\in\mathbb H_h$ at time
  $t_n$ and a given new time $t_{n+1}=t_n+\tau_n$ an update is computed as
  follows:
  \begin{itemize}
    \item[--]
      Compute a velocity field $\vec v^n_h\in\mathbb V_h^d$ and internal
      energy per unit volume $\inte^n_h\in\mathbb V_h$ by setting the
      corresponding nodal degrees of freedom to
      \begin{align*}
        \vec v_{i,K}^n := \vec m_{i,K}^n / \rho_{i,K}^n,
        \qquad
        \inte_{i,K}^n := \mathcal E_{i,K}^n - \tfrac12\big|\vec
        m_{i,K}^n\big|_{\ell^2}^2 / \rho_{i,K}^n.
      \end{align*}
    \item[--]
      Compute an update $\vec v^{n+1}_h\in\mathbb V_h^{d}$ and
      $\varphi^{n+1}_h\in\mathbb H_h$ for time $t_{n+1}$ by solving
      \eqref{eq:fully_discrete_theta_scheme_condensed}.
    \item[--]
      Reconstruct the momentum $\vec m^{n+1}_h\in\mathbb V_h^d$ and total
      mechanical energy $\mathcal E_{h}^{n+1} \in\mathbb V_h$ by setting the
      corresponding nodal degrees of freedom to
      \begin{align*}
        \vec m_{i,K}^{n+1} := \rho_{i,K}^n\,\vec m_{i,K}^{n+1},
        \qquad
         \mathcal E_{i,K}^{n+1} := \inte_{i,K}^n \,+\,
         \tfrac12\rho_{i,K}^n\big|\vec v_{i,K}^{n+1} \big|_{\ell^2}^2.
      \end{align*}
    \item[--] Set $\rho_h^{n+1} = \rho_h^{n}$.
  \end{itemize}
\end{definition}
\begin{corollary}
  The update procedure outlined in Definition~\ref{defi:source_update}
  maintains an energy inequality, i.\,e.,
  \begin{multline}
    \label{eq:fully_discrete_total_energy_balance}
    \sum_{j=1}^{\mathcal{N}}\sum_{K\in\mathcal{T}_h}\frac{1}{2}m_{j,K}
    \mathcal E_{i,K}^{n+1}
    \,+\,
    \frac1{2\alpha}\int_\domain\big|\nabla\varphi_h^{n+1}
    \big|_{\ell^2}^2\,\mathrm{d}x
    \\
    \le\;\sum_{j=1}^{\mathcal{N}}\sum_{K\in\mathcal{T}_h}\frac{1}{2}m_{j,K}
    \mathcal E_{i,K}^{n}
    \,+\, \frac1{2\alpha}\int_\domain\big|\nabla\varphi_h^{n}
    \big|_{\ell^2}^2\,\mathrm{d}x ,
  \end{multline}
  with equality holding for the case of $\theta=1/2$. Furthermore, by
construction the density $\rho^n_{i,K}$ and the internal energy per unit volume
$\inte^n_{i,K}$ are not modified.
\end{corollary}

\subsection{Restart strategies for the Gauß law}
The numerical scheme for the source update
\eqref{eq:fully_discrete_theta_scheme_condensed} does not preserve the Gauß
law. That is, a discrete counterpart of $-\Delta\varphi=\alpha\rho$ does
not necessarily hold true any more for the final update $\{\rho_h^{n+1},
\varphi_h^{n+1}\}$ obtained after a Strang split
(Definition~\ref{def:strang_splitting}), or, equivalently, a high-order
IMEX step. In this section, we follow a strategy outlined in \cite[Section
4]{eulerpoisson23} to postprocess the potential $\varphi_h^{n+1}$ in order
to reestablish a discrete Gauß law between $\varphi_h^{n+1}$ and
$\rho_h^{n+1}$. One possible restart strategy is to fully reset the
potential $\varphi_h^{n+1}$ at the end of the Strang (or IMEX) step:
\begin{definition}[Full Gauß law restart]
  \label{def:full_restart}
  Let $\{\rho_h^{n+1}, \varphi_h^{n+1}\}$ be the final density and
  potential obtained after a full strang split
  (Definition~\ref{def:strang_splitting}), or an equivalent IMEX update
  procedure. Then, compute $\widetilde \varphi_h^{n+1}\in\mathbb H_h$
  satisfying the discrete Gauß law:
  \begin{equation}
    \label{eq:discrete_gauss_law}
    (\nabla\widetilde\varphi_h^{n+1},\nabla\omega_h)
    =
    \alpha\left\langle\rho_h^{n+1},\omega_h\right\rangle_h
    \quad \text{for all}\; \omega_h \in \mathbb H_h,
  \end{equation}
  and set $\varphi_h^{n+1} \leftarrow \widetilde \varphi_h^{n+1}$.
\end{definition}
This process comes at the expense of losing the energy stability
established in Proposition
\ref{prop:fully_discrete_parabolic_energy_balance}. Fortunately, the next
result shows that the deviation introduced by this restart strategy is
small.
\begin{proposition}
  Assume that the hyperbolic update is at least first-order accurate in
  time and space and subject to a hyperbolic CFL condition $\tau_n =
  \mathcal O(h)$. Then, introducing the \emph{Gauß law residual}
  $\mathcal{R}_h^{n+1}[\omega_h] \,=\, \alpha\,
  \langle\rho_h^{n+1},\omega_h\rangle_h
  -(\nabla\Epot_h^{n+1},\nabla\omega_h)$ we have that
  \begin{align}
    \label{eq:gauss_law_residual}
    \big\|\mathcal{R}_h^{n+1}[\,.\,] -
    \mathcal{R}_h^{n}[\,.\,]\big\|_{\mathrm{op}}\;=\; \alpha\, \mathcal
    O\big(h^2+h^{1+s}\big),
  \end{align}
  where $s$ is the constant introduced in Assumption~\ref{ass:regular_mesh}
  and
  \begin{equation*}
    \|\mathcal R_h\|_{\mathrm{op}} :=
    \sup_{\substack{\omega_h \in \mathbb H_h \\ \omega_h \neq 0}}
    \frac{|\mathcal R_h[\omega]|}{\|\omega_h\|_{H^1(D)}}
  \end{equation*}
  for operators $\mathcal R_h : \mathbb H_h \to \mathbb R$.
\end{proposition}

\begin{proof}
  Let $\big\{[\rho^n_h, \vec m^n_h, \mathcal{E}^n_h], \varphi_h^n\big\}$ be
  the computed state at time $t_n$ and let the discrete state
  $\big\{[\rho^{n+1}_h, \vec m^{n+1}_h, \mathcal{E}^{n+1}_h],
  \varphi_h^{n+1}\big\}$ be the final update obtained after a full Strang
  split (Definition~\ref{def:strang_splitting}). Let $\hat\rho_h^{n}$ be
  the intermediate density update with which the source update is performed
  in the Strang split, and let $\breve {\vec v}^{n+\theta}_h$ be the
  intermediate velocity obtained from the theta scheme for the source
  update; see Definition~\ref{def:strang_splitting} and
  \eqref{eq:fully_discrete_theta_scheme_condensed}. We recall that
  $\varphi^{n+\theta}=\theta\varphi^{n+1}+(1-\theta)\varphi^{n}$ and by
  slight abuse of notation we also introduce a density
  $\rho^{n+\theta}:=\theta\rho^{n+1}+(1-\theta)\rho^{n}$. Recalling
  \eqref{eq:fully_discrete_theta_scheme_symmetric}, we then have
  \begin{align*}
    \mathcal{R}_h^{n+\theta}[\omega_h] - \mathcal{R}_h^{n}[\omega_h]
    \,&=\, \alpha\,
    \langle\rho_h^{n+\theta}-\rho_h^{n},\omega_h\rangle_h
    -(\nabla\varphi_h^{n+\theta}-\nabla\varphi_h^{n},\nabla\omega_h)
    \\
    \,&=\, \alpha\,
    \langle\rho_h^{n+\theta}-\rho_h^{n},\omega_h\rangle_h
    - \theta\tau_n\alpha\, \big\langle\hat\rho_h^n \breve{\vec
    v}_h^{n+\theta}, \nabla \omega_h \big\rangle_h,
  \end{align*}
  We observe that the right-hand side of the above equation resembles a
  discretization of the balance of mass equation $\partial_t \rho + \nabla
  \cdot \vec m = 0$ with a $\theta$ time stepping scheme:
  \begin{align*}
    \langle\tilde\rho_h^{n+\theta},\omega_h\rangle_h + \theta\tau_n\alpha\,
    \big(\tilde\rho_h^{n+\theta} \breve{\vec v}_h^{n+\theta}, \nabla
    \omega_h \big) = \langle\rho_h^{n},\omega_h\rangle_h,
  \end{align*}
  where $\tilde\rho_h^{n+\theta}$ is a formal density update with a given
  (interpolatory) velocity field $\breve{\vec v}_h^{n+\theta}$ at time
  $t_{n+\theta}$. From our assumptions on the hyperbolic solver we now make
  the claim that $\|\tilde\rho_h^{n+\theta} - \rho_h^{n+\theta}\|=\mathcal
  O(\tau_n^2)=\mathcal O(h^2),$ which would need to be verified in detail
  for the concrete chosen method. The switch between the lumped inner
  product and the full $L^2$-product introduces an additional error of the
  order $\mathcal O(\tau_n h^s)=\mathcal O(h^{1+s})$.
\end{proof}

\begin{remark}
  The result can be improved to $\big\|\mathcal{R}_h^{n+1}[\,.\,] -
  \mathcal{R}_h^{n}[\,.\,]\big\|_{\mathrm{op}}= \alpha\, \mathcal
  O(h^3+h^{1+s})$ for the case of a Strang split with a hyperbolic update
  of at least second order and a source update with $\theta = 1/2$.
\end{remark}
\begin{corollary}\label{cor:gauss_law_residual}
  Let $\widetilde \varphi_h^{n+1}$ be the solution to
  \eqref{eq:discrete_gauss_law} and $\varphi_h^{n+1}$ the solution to
  \eqref{eq:fully_discrete_theta_scheme_condensed}.
  Now, observing that $\mathcal R_h^{n}[\tilde \varphi_h^{n} -
  \varphi_h^{n}]= \|\nabla \tilde \varphi_h^{n} -
  \nabla\varphi_h^{n}\|_{\Ltwod}^2$ and using
  \eqref{eq:gauss_law_residual} in a telescopic sum establishes
  \begin{align*}
    \|\nabla \tilde \varphi_h^{n} - \nabla\varphi_h^{n}\|_{\Ltwod}
    &=\,(t_n-t_0)\,\alpha\,\mathcal{O}\big(h+h^s\big),
    \\
    \|\nabla\widetilde{\Epot}_h^{n+1}\|^2_{\Ltwod}
    - \|\nabla\Epot_h^{n+1}\|^2_{\Ltwod}
    &=\,(t_n-t_0)\,\alpha\,\mathcal{O}\big(h+h^s\big).
  \end{align*}
  Similarly, higher-order convergence of $\mathcal{O}(h^2+h^s)$ holds true
  for the case of a Strang split with a hyperbolic update of at least
  second order and a source update with $\theta=1/2$.
\end{corollary}
It is desirable to reestablish a discrete energy balance again after performing
restart \cite{eulerpoisson21, eulerpoisson22, eulerpoisson23}. This can
sometimes be achieved via an artificial relaxation approach by lowering
the kinetic energy appropriately~\cite{eulerpoisson23}. We summarize:
\begin{definition}[Gauß law restart with relaxation]\label{def:relaxation}
  Let $\{\rho_h^{n+1}, \varphi_h^{n+1}\}$ be the final density and
  potential obtained after a full Strang split
  (Definition~\ref{def:strang_splitting}), and let $\widetilde
  \varphi_h^{n+1}\in\mathbb H_h$ be the solution to
  \eqref{eq:discrete_gauss_law}. Introduce
  \begin{align*}
    \delta\mathcal{E}^{n+1} \;&:=\; \frac{1}{2\,\alpha}
    \Big(\|\nabla\tilde\varphi_h^{n+1}\|_{\Ltwod}^2 -
\|\nabla\varphi_h^{n+1}\|_{\Ltwod}^2\Big)
    \\[0.125em]
    \mathcal{E}_{\mathrm{kin}}^{n+1} \;&:=\;
    \sum_{i} \frac{m_i}{2}\rho_i^n\big|\vec v^{n+1}\big|_{\ell^2}^2,
  \end{align*}
  and define the relaxation coefficient as
  \begin{equation*}
    r^{n+1} :=
    \begin{cases}
      \sqrt{
        1 - \frac{\delta \mathcal E^{n+1}}{\mathcal E_{\mathrm{kin}}^{n+1}}
      }
      & \text{ if } 0 < \delta \mathcal E^{n+1} \leq
      \mathcal E_{\mathrm{kin}}^{n+1}, \\[1em]
      1 & \text{else}.
    \end{cases}
  \end{equation*}
  Then, set
  \begin{align*}
    \varphi_h^{n+1} \;\leftarrow\; \widetilde \varphi_h^{n+1},
    \qquad
    \vec m_i^{n+1} \;\leftarrow\;
    r^{n+1}\,\vec m_i^{n+1}.
  \end{align*}
\end{definition}
\begin{remark}[Relaxation]\label{rem:relaxation}
  Relaxation of the kinetic energy only occurs if \\restarting the Gauß law
  leads to an increase in potential energy which does not exceed the amount
  of kinetic energy in the system. In systems with large electromagnetic
  forces that balance out to yield relatively slow macroscopic motion, such
  as the diocotron instability in Section~\ref{subse:numerical_diocotron},
  or in systems where restarting the potential lowers the potential energy,
  relaxation does not occur.
\end{remark}


\section{Numerical illustrations}
\label{sec:numerical}
We now present a number of computational results that illustrate the
performance of the method (a) on a smooth test case (see
Section~\ref{subse:numerical_smooth}, and (b) on a problem in the magnetic
drift limit with huge scale separation of the time scales (see
Section~\ref{subse:numerical_diocotron}).

\subsection{Implementational details}
\label{subse:implementation}

The numerical algorithms discussed above have been implemented in the
hydrodynamic solver framework \texttt{ryujin}~\cite{euler21,
navierstokes22}, which is based on the finite element library
\texttt{deal.II}~\cite{dealII96, dealIIcanonical}. Our implementation is
freely available
online\footnote{\url{https://github.com/conservation-laws/ryujin}} under a
permissible open source
license.\footnote{\url{https://spdx.org/licenses/Apache-2.0.html}} Details
on the invariant-domain preserving solver for the hyperbolic subsystem
\eqref{eq:hyperbolic_operator} as discussed in Appendix~\ref{sec:HypSolver}
can be found in \cite{euler21}, the discretization with discontinuous
finite elements is discussed in \cite{euler25}.

The source update operator in turn is discretized with \texttt{deal.II}'s
matrix-free operator framework~\cite{Kronbichler2012, Kronbichler2018,
Kronbichler2019b}: The action of the stiffness matrix described by the
bilinear form \eqref{eq:discrete_bilinear_form} is implemented by
repeatedly reconstructing the local stencil of the system matrix whenever a
matrix-vector multiplication is performed. We refer the reader to
\cite{navierstokes22} for a related, detailed discussion and cost analysis
of the matrix-free operator approach applied to such source systems. We
note that the dominant operator in \eqref{eq:discrete_bilinear_form} is a
Laplacian and not a mass matrix, irrespective of the chosen time-step size.
This is, for example, in contrast to the parabolic diffusion operator
encountered in the compressible Navier-Stokes equations when performing an
operator split~\cite{navierstokes22}. In addition, the boundedness constant
of \eqref{eq:discrete_bilinear_form} scales with
$\mathcal{O}(\tau_n\omega_p^2)$; see Proposition~\ref{prop:boundedness}.
For these reasons we settled on using a geometric multigrid preconditioner
with Chebyshev smoother \cite{Kronbichler2018,Clevenger2021}. After careful
tuning we observe that solving the source update
\eqref{eq:fully_discrete_theta_scheme_condensed} has the same computational
cost as performing the explicit hyperbolic update: For the numerical
results discussed below in Section~\ref{subse:numerical_diocotron} we spend
around 28\,\% of the wall time in the source update, while the remaining
72\,\% of the wall time was split between two hyperbolic updates (with a
third order Runge-Kutta scheme), vector transfer and solution output.

\subsection{Isentropic vortex}
\label{subse:numerical_smooth}
We consider the smooth isentropic vortex solution adapted to the case of
the Euler-Poisson system \eqref{eq:magnetic_euler_poisson} with
$\vec\Omega=\vec 0$; see \cite[Definition~5.2]{eulerpoisson23}.
The isentropic vortex is an analytic solution of the Euler-Poisson system when
the system is augmented with a background density that exactly cancels the
effect of the density distribution leading to $\varphi = \mathrm{const}_x$.

The exact solution is given by \cite[Definition~5.2]{eulerpoisson23}:
\begin{align*}
  \boldsymbol{r}(\xcoord, t) &:= \xcoord - \xcoord_0 - M t,
  \quad
  f(\xcoord, t) := \tfrac{\beta}{2 \pi}
  e^{\frac{1}{2} (1 - |\boldsymbol{r}|^2)},
  \quad
  T(\xcoord,t) := 1 - \tfrac{\gamma - 1}{2 \gamma} f^2,
\end{align*}
with vortex speed $M=[1,1]^T$, $\gamma = 5/3$, vortex size $\beta=5$,
coupling constant $\alpha = 1$, and magnetic field density $\vec\Omega=\vec
0$. The domain is $\domain=[-5,5]^2$, and the initial center of the vortex
is $\xcoord_0 = [-1,-1]^T\in\domain$. We approximate the solution to
\eqref{eq:magnetic_euler_poisson} using our numerical method on a sequence
of successively refined uniform meshes for $\domain$. For initial
conditions, we interpolate the exact solution at time $t = 0$ to the finite
element spaces $\mathbb V_h$, $\mathbb H_h$ defined in Section
\ref{sec:finite_elements}. For boundary conditions, we impose
non-homogeneous Dirichlet boundary conditions using the exact solution.

We construct a low order method by using a forward Euler step for the
hyperbolic part and a backward Euler step for the parabolic part (by
setting $\theta=1$ in \eqref{eq:fully_discrete_theta_scheme_condensed}).
Correspondingly, a high order method is constructed by combining an
optimal, third order explicit Runge-Kutta time stepping scheme
\cite{Ern_Guermond_2022} for the hyperbolic part and the Crank-Nicoloson
scheme (by setting $\theta=1/2$ in
\eqref{eq:fully_discrete_theta_scheme_condensed}) for the source update in
a Strang split. The CFL parameter for the hyperbolic update is set to
$0.1$; see \cite{GuePoTom2019}. At the final time step $t_N := t_F$, we
compute the $L^1$ error
\begin{align*}
  \delta_h := \|\rho(t_F) - \rho_h^N\|_{L^1(\domain)} + \|\vec m(t_F) - \vec
  m_h^N\|_{\vec L^1(\domain)} + \|\mathcal E(t_F) - \mathcal
E_h^N\|_{L^1(\domain)}.
\end{align*}
For our simulations, we set $t_F := 1$. We report the results in
Table~\ref{tab:isentropic_vortex} for both the low-order and high-order method.
We recover the expected optimal convergence rates.
\begin{table}[t]
  \centering
  \subfloat[Low order method]{%
  \begin{tabular}{@{}rcr@{}}
    \toprule
    {\bfseries dofs} & {\bfseries $\boldsymbol{\delta_h}$} &
    {\bfseries rate} \\[0.25em]
    4096   & 1.220e+01 & \multicolumn{1}{l}{} \\
    16384  & 7.366e+00 & 0.73 \\
    65536  & 4.091e+00 & 0.85 \\
    262144 & 2.165e+00 & 0.91 \\
    \bottomrule
  \end{tabular}}
  \hspace{1em}
  \subfloat[High order method]{%
  \begin{tabular}{@{}rcr@{}}
    \toprule
    {\bfseries dofs} & {\bfseries $\boldsymbol{\delta_h}$} &
    {\bfseries rate} \\[0.25em]
    4096   & 3.827e-01 & \multicolumn{1}{l}{} \\
    16384  & 9.118e-02 & 2.07 \\
    65536  & 2.210e-02 & 2.04 \\
    262144 & 5.438e-03 & 2.02 \\
    \bottomrule
  \end{tabular}}

  \caption{%
    \label{tab:isentropic_vortex}
    Smooth isentropic vortex test: $L^1$-error $\delta_h$ at the final time
    $t_F$ for varying global refinement levels and corresponding
    convergence rate.}
\end{table}

\subsection{Diocotron instability}
\label{subse:numerical_diocotron}
Next, we consider the diocotron instability \cite{Filbet2016II, Piao2018,
Hamiaz2016, Tsidulko2004, Petri2002, Petri2006, Petri2009} that arises in the
context of cold plasma physics. The diocotron instability occurs in an
experimental setup where an essentially hollow cylindrical beam of electrons is
aligned with an out-of-plane magnetic field. Such an electron profile
experiences an instability caused by \emph{$\vec E\times\vec B$
drift-rotational shear}. We use model \eqref{eq:magnetic_euler_poisson} with
the isothermal equation of state $p = \theta \rho$, where $\theta \geq 0$ is
the temperature of the plasma. This choice of closure is fundamental for the
design a meaningful computational experiment. Since $\theta$ is a model
parameter and the corresponding thermal sound speed is given by $c =
\sqrt{\theta}$, we can make thermal sound speed as small as we wish. For
instance, we can make the thermal sound speed negligible in comparison to
material velocity $\vec v$. Similarly, with a proper choice of the parameters
$\alpha$ and  $|\vec \Omega|_{\ell^2}$, we can make plasma time scales orders
of magnitude larger that the hydrodynamic time-scales. With this setup, we can
force, onto the magnetic Euler-Poisson model \eqref{eq:magnetic_euler_poisson},
the conditions required to operate in the magnetic-drift regime.

For the test case we set $d=2$, i.\,e., the motion of the plasma is
restricted to the $(x,y)$-plane, and the prescribed magnetic field $\vec B$ is
spatially uniform and transverse to the plasma; see
Section~\ref{subse:lorentz}. Let $0< r_0 < r_1 < R$ be three chosen radii
and let the computational domain $\domain\subset\mathbb{R}^2$ be the disc
centered in the origin with radius $R$. Choosing two background densities
$0<\rho_{\min}\ll\rho_{\max}$, introduce an initial density profile given
by an annulus of high density with radii $r_0$ and $r_1$:
\begin{align*}
  \rho_0(\vec x) :=
  \begin{cases}
    \begin{aligned}
    &\rho_{\min}
    &\quad&
    \text{for } |\vec x|_{\ell^2} \leq r_0 \text{ or }
    |\vec x|_{\ell^2} \geq r_1,
    \\[0.5em]
    &\rho_{\max}\,\delta_{\mathrm{per}}(\vec x)
    &\quad&
    \text{for } r_0 < |\vec x|_{\ell^2} < r_1.
    \end{aligned}
  \end{cases}
\end{align*}
Here,
\begin{align*}
  \delta_{\mathrm{per}}(\vec x)\,:=\,
  1 - \delta + \delta\,\sin\big(\ell \arctan(x_2/x_1)\big),
\end{align*}
is a small perturbation of the background density with chosen coefficients
$0 < \delta \ll 1/2$, and $\ell \in \mathbb N$. The initial potential
$\varphi_0(\vec x)$ is determined by the Gauß law,
\begin{equation*}
  -\Delta\varphi_0 = \alpha \rho_0,
\end{equation*}
and the initial velocity $\vec v_0(\vec x)=\vec v_{\mathrm{dr}}$ is given by
\eqref{eq:magnetic_drift_velocity}. We now introduce a scaling parameter
$\beta>0$ and set $\alpha:=\rho_{\max}^{-1}\beta^{2}$ and
$\omega_{\mathrm{c}}:=\beta^{2}.$ This results in the following relevant time
scales:
\begin{align*}
  \omega_{\mathrm{c}} \,\sim\, \beta^{2},
  \qquad
  \omega_{\mathrm{p}} \,\sim\, \sqrt{\rho_{\max}\alpha} = \beta^{1},
  \qquad
  \omega_{\mathrm{d}} \,\sim\,1.
\end{align*}

\paragraph{Computational setup and initial values}
We now fix the following parameters
\begin{align*}
  r_0=6,\;r_1=8,\;R=16,
  \qquad
  \rho_{\min}=10^{-6},\;
  \rho_{\max}=1,
  \qquad
  \beta=10^{6},
  \qquad
  \delta=0.1,
\end{align*}
with the goal to perform a series of computations for initial values with
differing perturbation modes, $\ell\in\{3,4,5\}$. We choose a coarse
discretization of $\domain$ into 12 quadrilaterals and refine the mesh
globally to refinement levels $r=6,7,8$ and $9$. The corresponding number
of degrees of freedom per component, i.\,e., $\dim\mathbb{V}_h$, are:
196,608 for $r=6$; 786,432 for $r=7$; 3,145,728 for $r=8$; and 12,582,912
for $r=9$. It holds true that $\dim\mathbb{H}_h\approx\dim\mathbb{V}_h/4$.
We first set the initial density by
interpolating $\rho_0(\vec x)$:
\begin{align*}
  \rho_h^0(\vec x) := \sum_{j=1}^{\mathcal{N}} \sum_{K \in \mathcal T_h}
  \rho_{j,K}^0\phi_{j,K}^h(\vec x),
  \qquad
  \rho_{j,K}^0 := \rho_0(\vec x_j).
\end{align*}
Then we determine the initial potential $\varphi_h^0\in\mathbb{H}_h$ by
solving the discrete Gauß law:
\begin{align*}
  \big(\nabla\varphi_h^0,\nabla\psi_h\big)_{\Ltwod}
  &= \alpha\,\langle \rho_h^0, \psi_h \rangle_h,
  \qquad
  \text{for all }\psi_h\in\mathbb H_h,
\end{align*}
and with homogeneous Dirichlet conditions enforced on the discrete boundary
of $\mathcal T_h$. Finally, the initial velocity is set to
\begin{align*}
  \vec v_h^0(\vec x) := \sum_{j=1}^{\mathcal{N}} \sum_{K \in \mathcal T_h}
  \vec v_{j,K}^0\phi_{j,K}^h(\vec x),
  \qquad
  \vec v_{j,K}^0 :=
  -\left.\left( \frac{\nabla \varphi_h^0 \times \vec \Omega} {|\vec
  \Omega|_{\ell^2}^2} \right)\right|_K (\vec x_{K,j}).
\end{align*}
We run the simulations to final time $t_f = 10$, which corresponds to 10
periods of the diocotron timescale. A series of temporal snapshots
outlining the time evolution of the solution for refinement level $r=9$ are
shown in Figures~\ref{fig:3-mode}, \ref{fig:4-mode}, and~\ref{fig:5-mode}.
In all simulations, we observe a time-step size on the order of $10^{-3}$
to $10^{-4}$, which is 2 to 3 orders of magnitude larger than the plasma
timescale and 8 to 9 orders of magnitude larger than the cyclotron
timescale. As the figures show, our method can successfully overstep the
fast-moving plasma and cyclotron dynamics while capturing the slow-moving
diocotron instabilities. In particular, our numerical results agree well
visually with existing results in the literature; cf. \cite[Figure
7.21]{Crockatt2021}, \cite{Crockatt2022}, \cite[Figure 9]{LaSpina2024},
and \cite[Figure 6.1]{eulerpoisson22}.
\begin{figure}[tp]
  \begin{center}
    \setlength\fboxsep{0pt}
    \setlength\fboxrule{0.5pt}
    \subfloat[$t=0.01\,t_f$]{\fbox{\includegraphics[width=40mm]{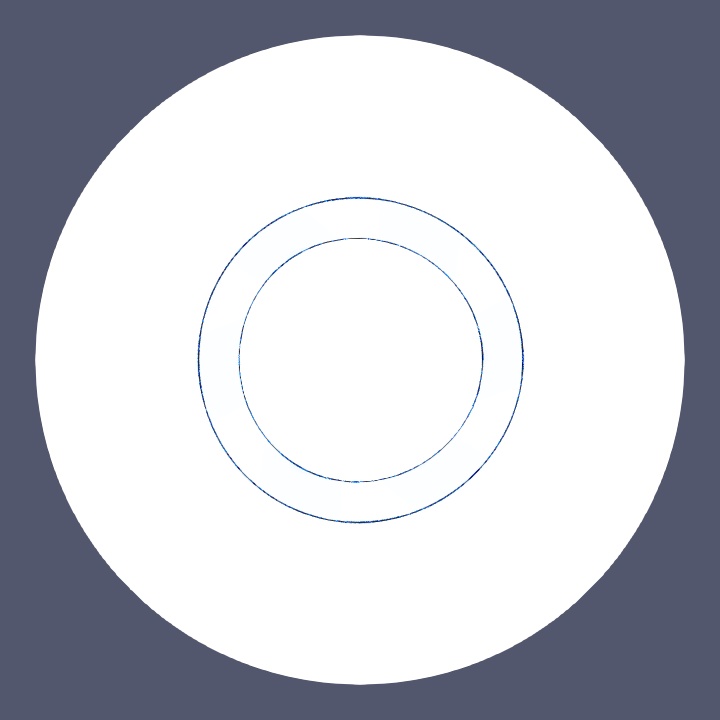}}}\;
    \subfloat[$t=1/8\,t_f$]{\fbox{\includegraphics[width=40mm]{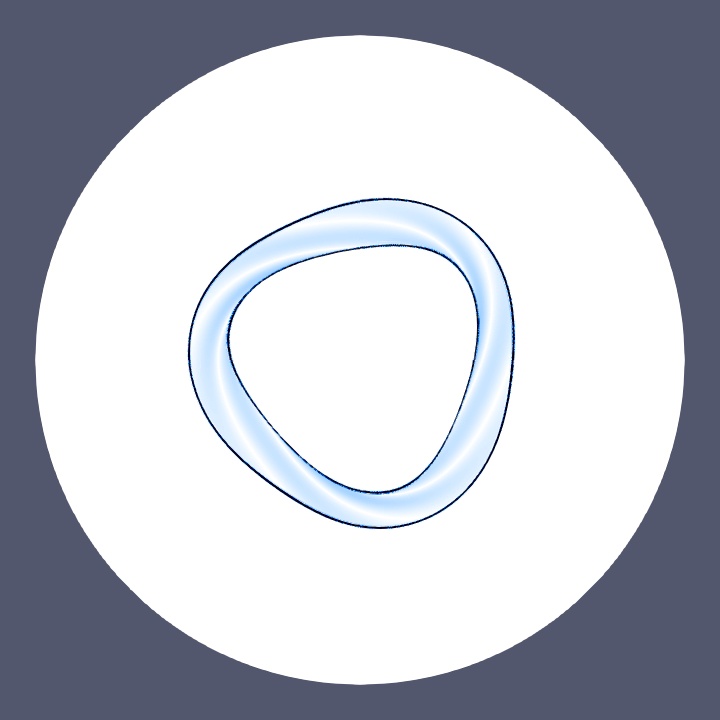}}}\;
    \subfloat[$t=2/8\,t_f$]{\fbox{\includegraphics[width=40mm]{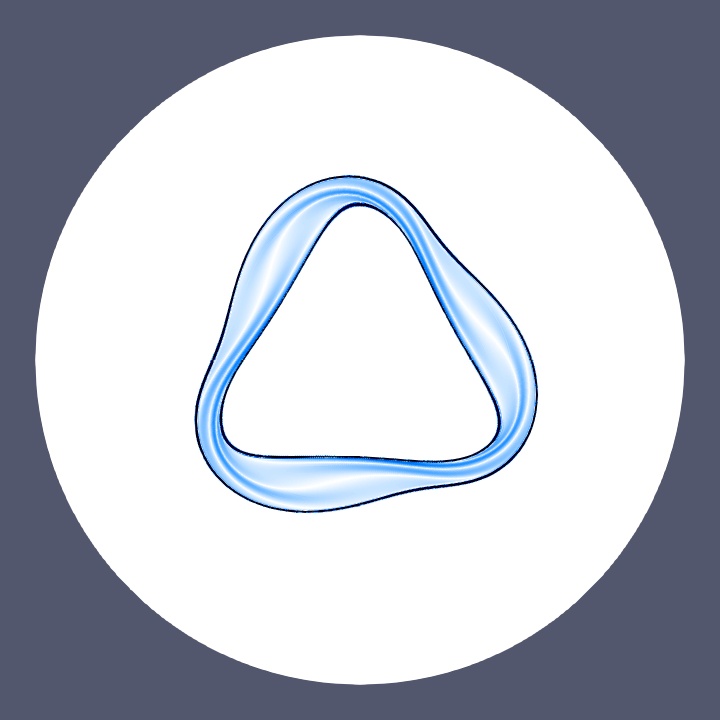}}}
    \vspace{-0.75em}
    \subfloat[$t=3/8\,t_f$]{\fbox{\includegraphics[width=40mm]{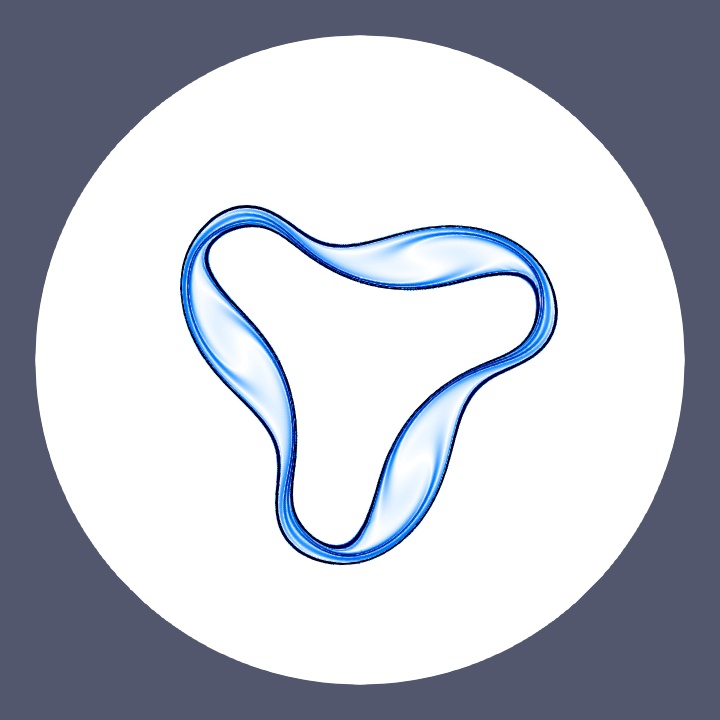}}}\;
    \subfloat[$t=4/8\,t_f$]{\fbox{\includegraphics[width=40mm]{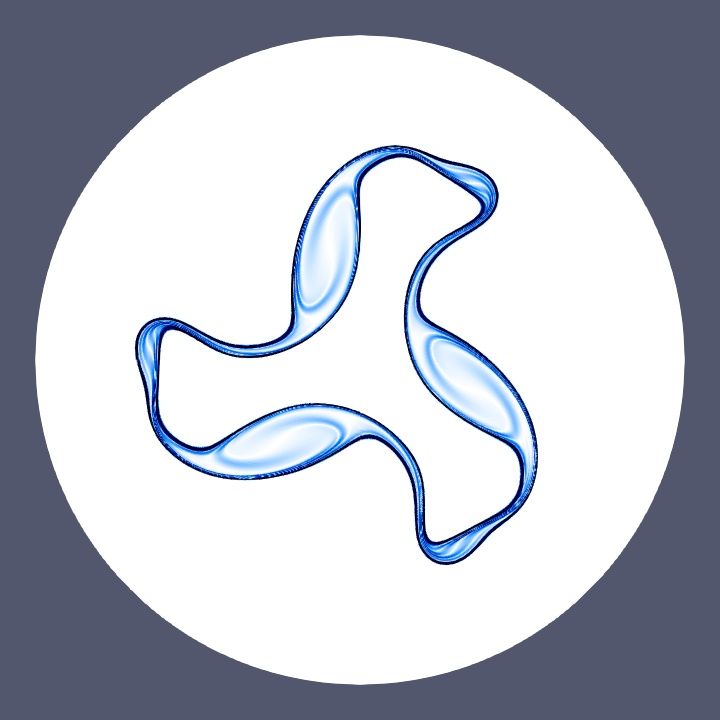}}}\;
    \subfloat[$t=5/8\,t_f$]{\fbox{\includegraphics[width=40mm]{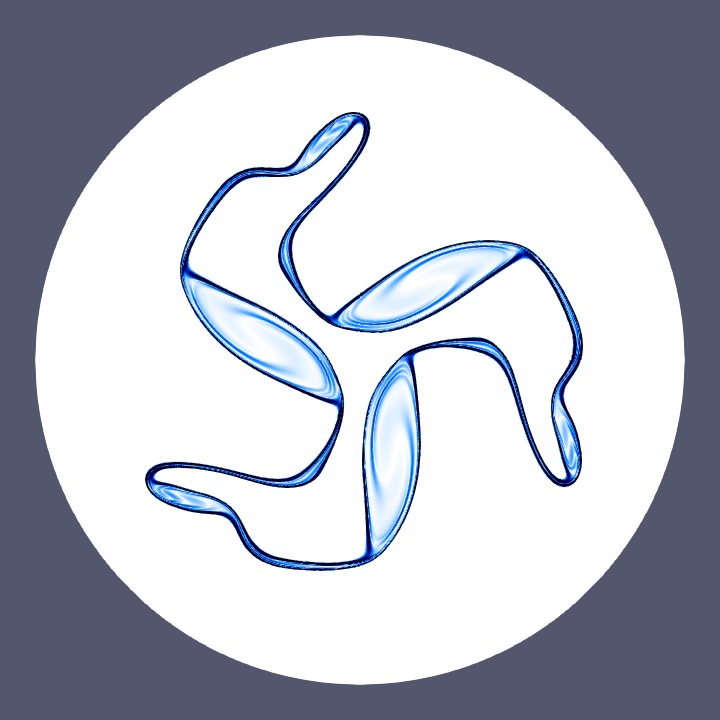}}}
    \vspace{-0.75em}
    \subfloat[$t=6/8\,t_f$]{\fbox{\includegraphics[width=40mm]{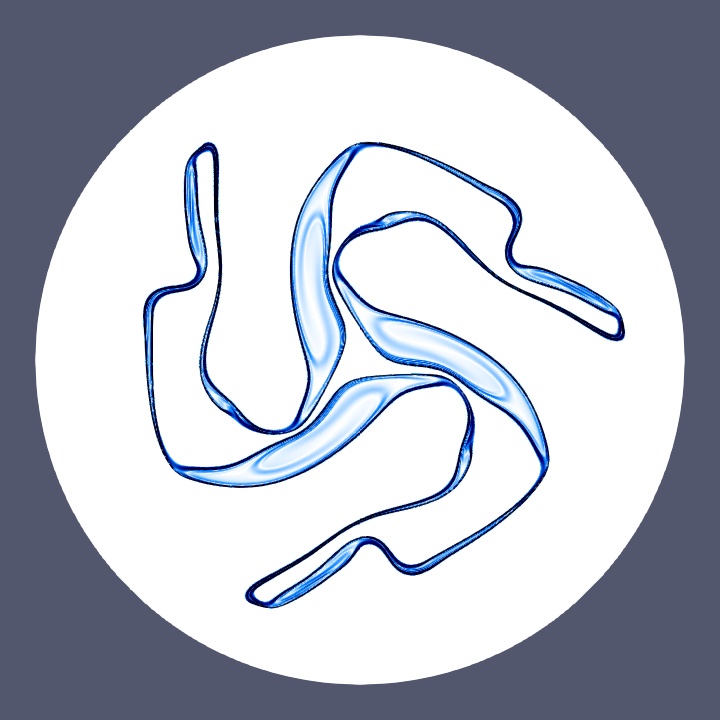}}}\;
    \subfloat[$t=7/8\,t_f$]{\fbox{\includegraphics[width=40mm]{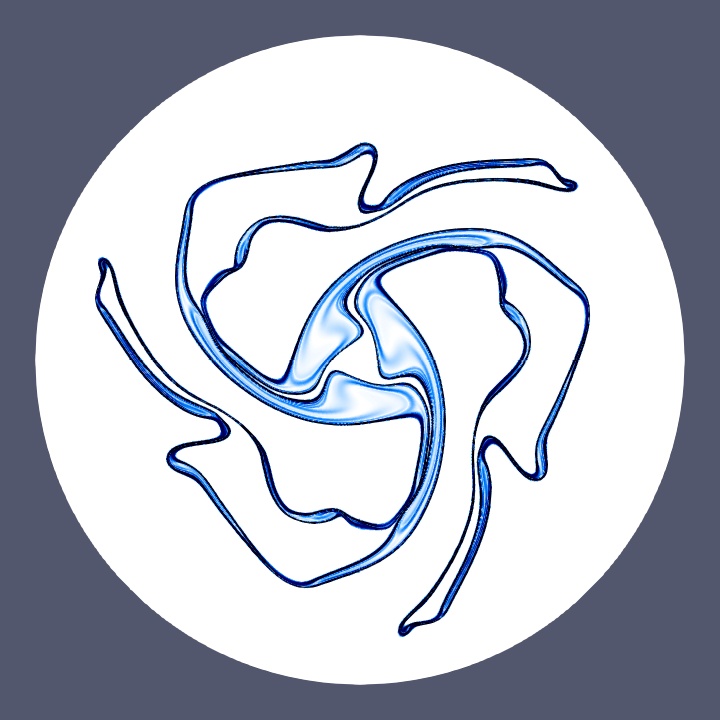}}}\;
    \subfloat[$t=t_f$]{\fbox{\includegraphics[width=40mm]{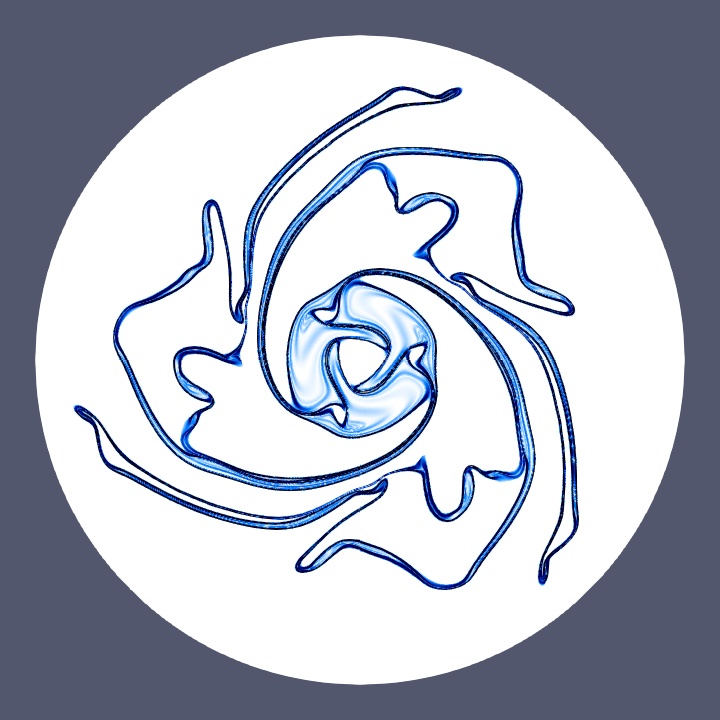}}}
  \end{center}
  \caption{\label{fig:3-mode}%
    Temporal snapshots of a schlieren plot of the density profile of the
    third mode diocotron instability test case. Reference computation with
    no restart on refinement level $r=9$ amounting to 12,582,912 dG degrees
    of freedom per component. Here, $t_f = 10$.}
\end{figure}
\begin{figure}[tp]
  \begin{center}
    \setlength\fboxsep{0pt}
    \setlength\fboxrule{0.5pt}
    \subfloat[$t=0.01\,t_f$]{\fbox{\includegraphics[width=40mm]{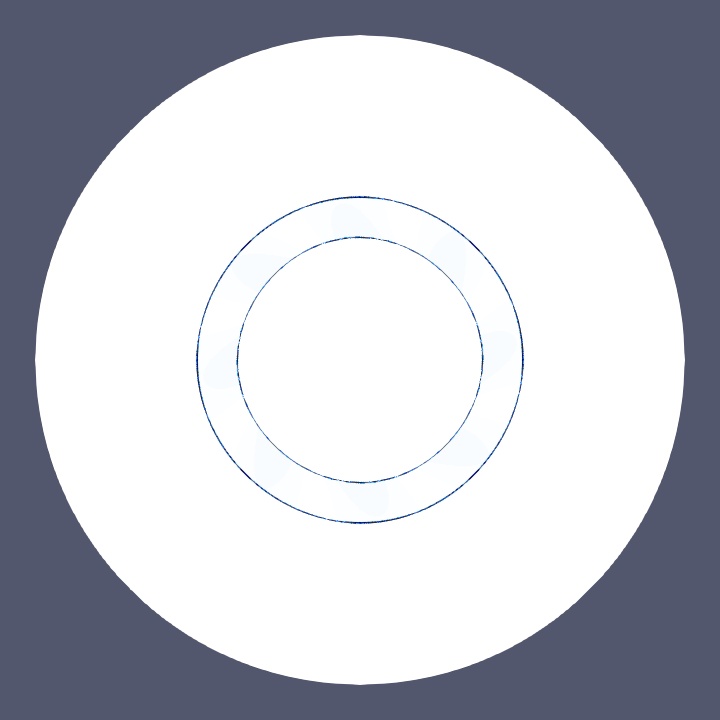}}}\;
    \subfloat[$t=1/8\,t_f$]{\fbox{\includegraphics[width=40mm]{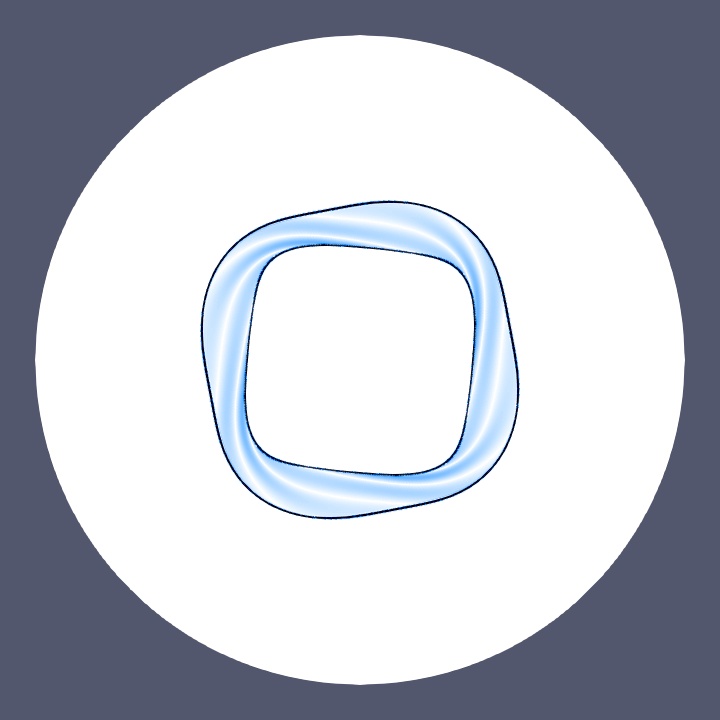}}}\;
    \subfloat[$t=2/8\,t_f$]{\fbox{\includegraphics[width=40mm]{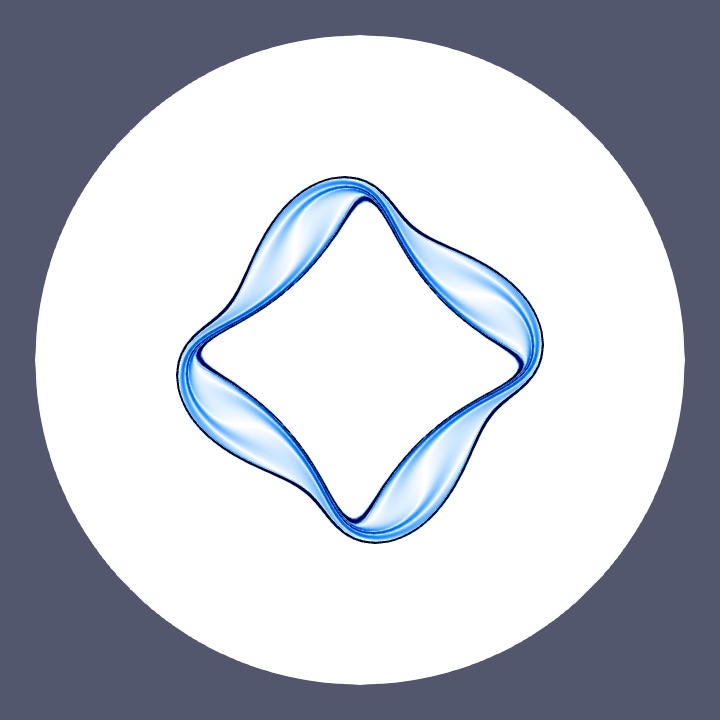}}}
    \vspace{-0.75em}
    \subfloat[$t=3/8\,t_f$]{\fbox{\includegraphics[width=40mm]{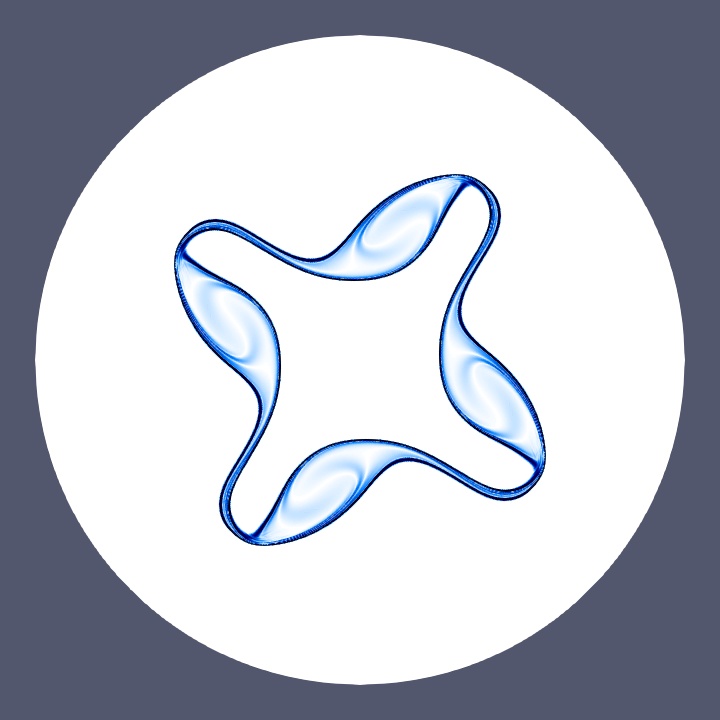}}}\;
    \subfloat[$t=4/8\,t_f$]{\fbox{\includegraphics[width=40mm]{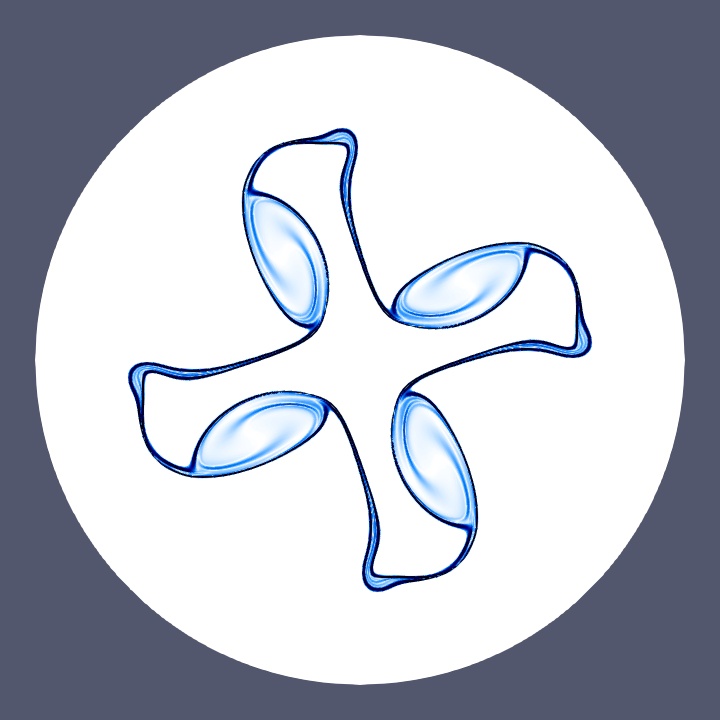}}}\;
    \subfloat[$t=5/8\,t_f$]{\fbox{\includegraphics[width=40mm]{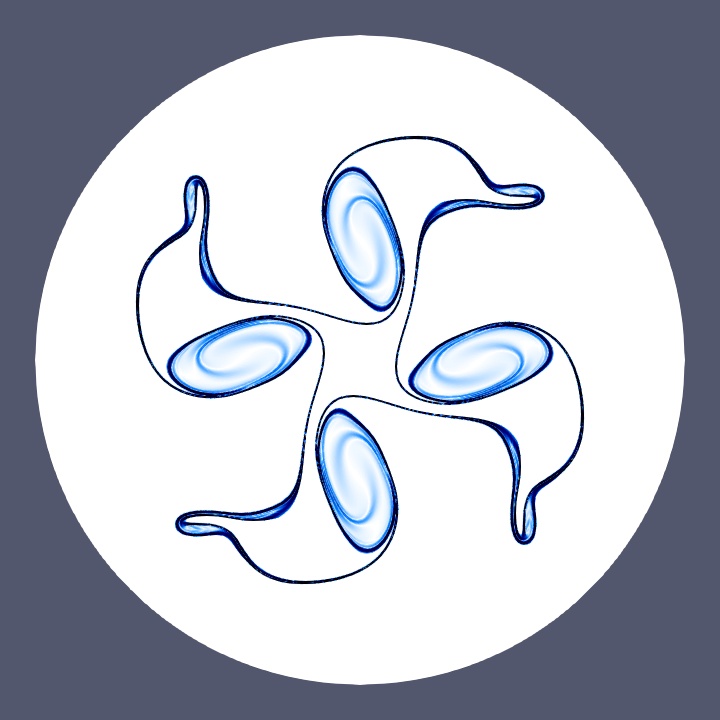}}}
    \vspace{-0.75em}
    \subfloat[$t=6/8\,t_f$]{\fbox{\includegraphics[width=40mm]{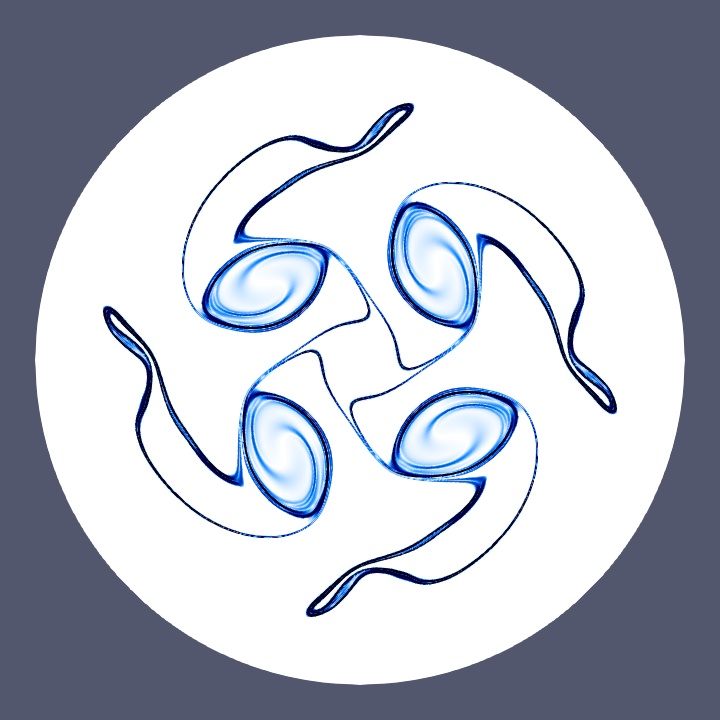}}}\;
    \subfloat[$t=7/8\,t_f$]{\fbox{\includegraphics[width=40mm]{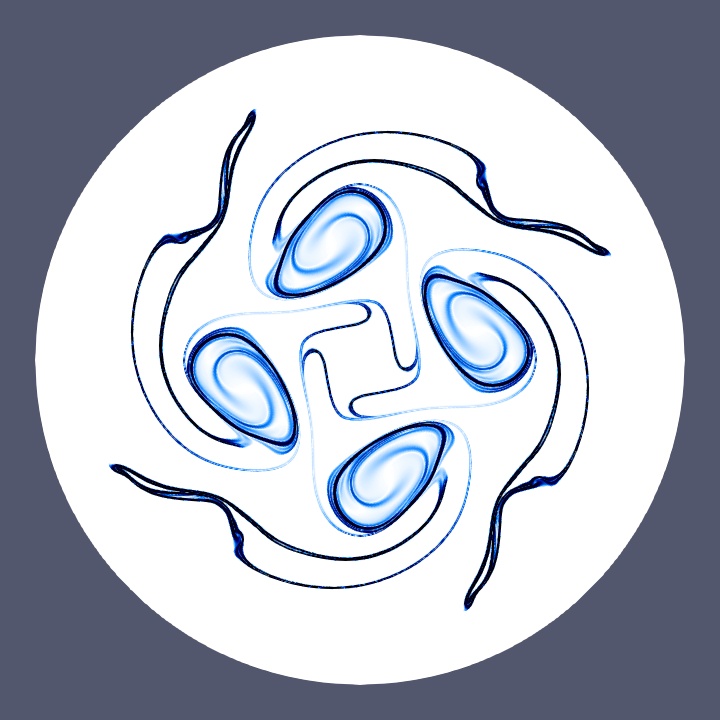}}}\;
    \subfloat[$t=t_f$]{\fbox{\includegraphics[width=40mm]{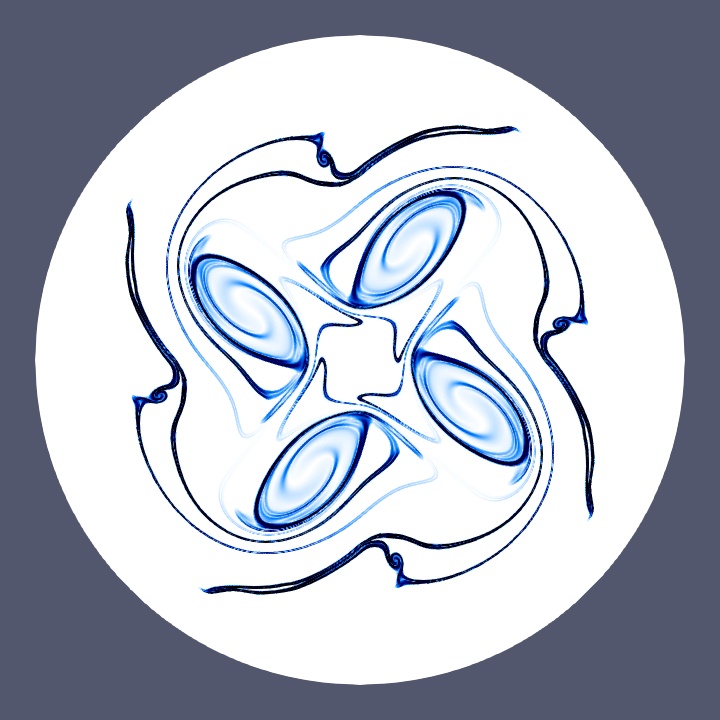}}}
  \end{center}
  \caption{\label{fig:4-mode}%
    Temporal snapshots of a schlieren plot of the density profile of the
    fourth mode diocotron instability test case. Reference computation with
    no restart on refinement level $r=9$ amounting to 12,582,912 dG degrees
    of freedom per component. Here, $t_f = 10$.}
\end{figure}
\begin{figure}[tp]
  \begin{center}
    \setlength\fboxsep{0pt}
    \setlength\fboxrule{0.5pt}
    \subfloat[$t=0.01\,t_f$]{\fbox{\includegraphics[width=40mm]{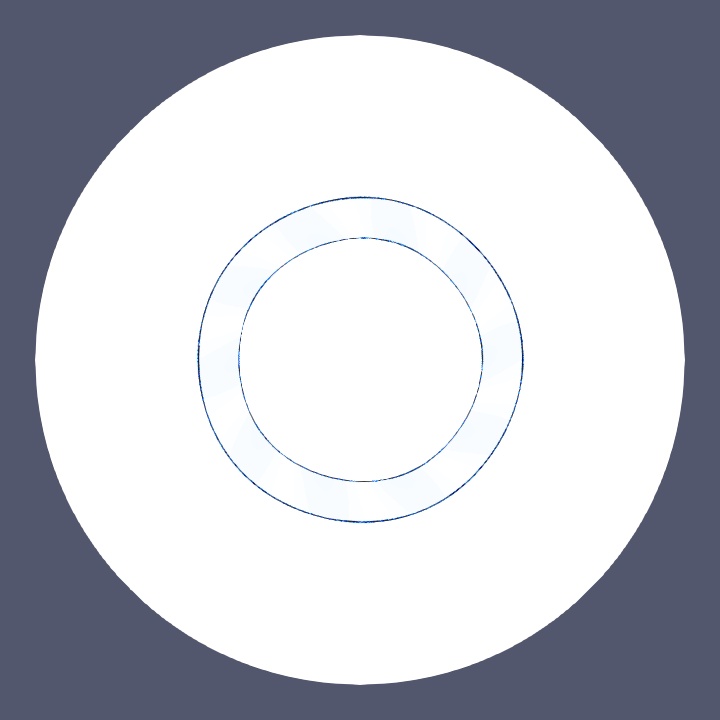}}}\;
    \subfloat[$t=1/8\,t_f$]{\fbox{\includegraphics[width=40mm]{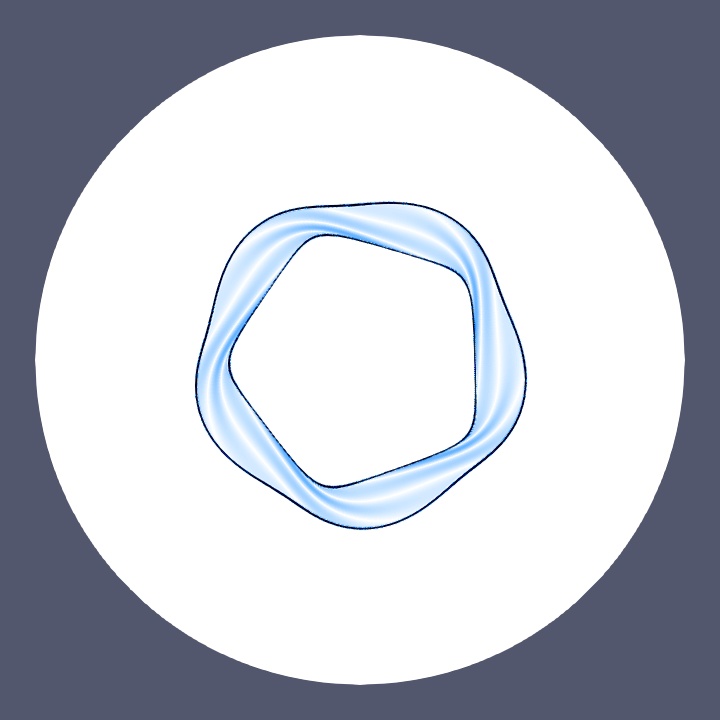}}}\;
    \subfloat[$t=2/8\,t_f$]{\fbox{\includegraphics[width=40mm]{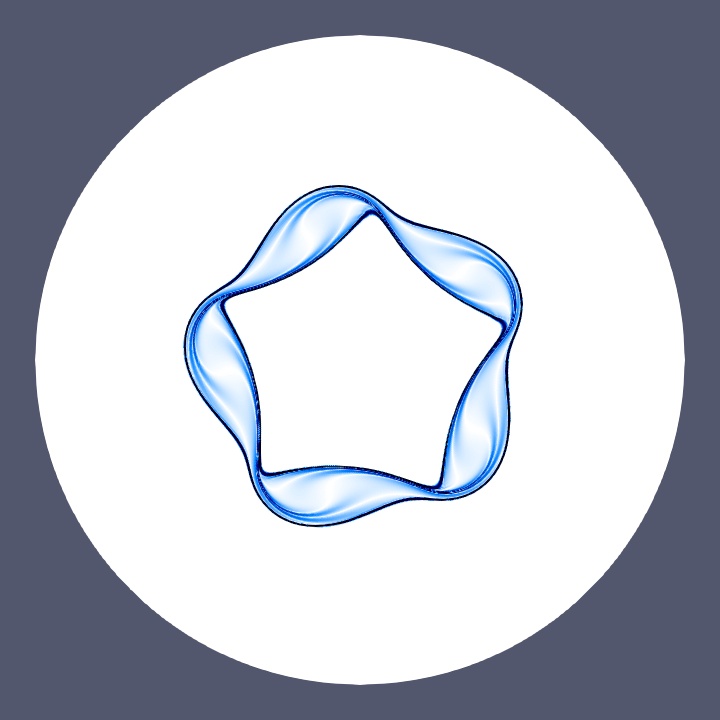}}}
    \vspace{-0.75em}
    \subfloat[$t=3/8\,t_f$]{\fbox{\includegraphics[width=40mm]{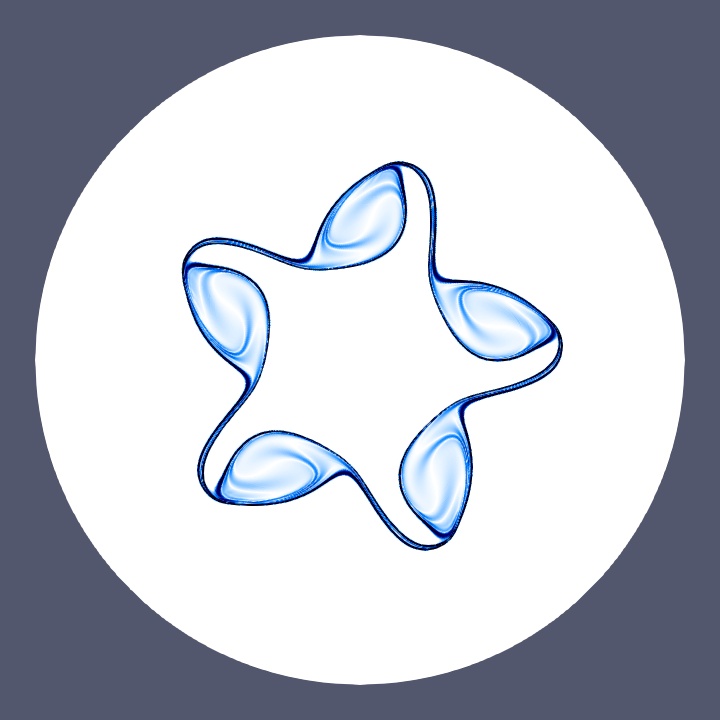}}}\;
    \subfloat[$t=4/8\,t_f$]{\fbox{\includegraphics[width=40mm]{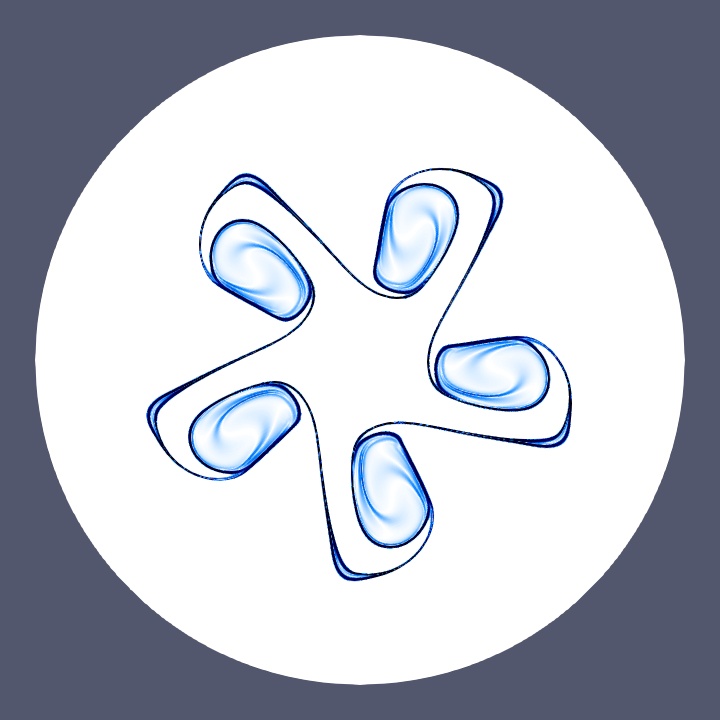}}}\;
    \subfloat[$t=5/8\,t_f$]{\fbox{\includegraphics[width=40mm]{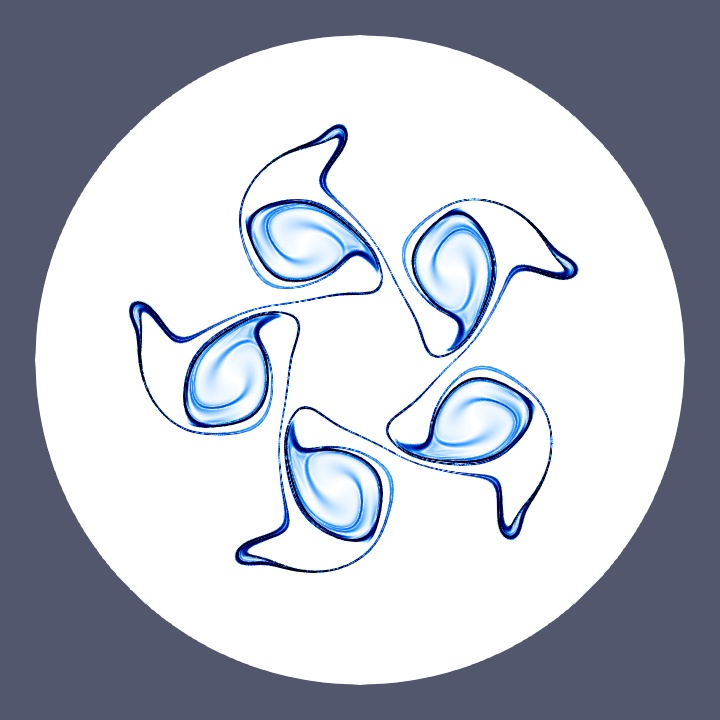}}}
    \vspace{-0.75em}
    \subfloat[$t=6/8\,t_f$]{\fbox{\includegraphics[width=40mm]{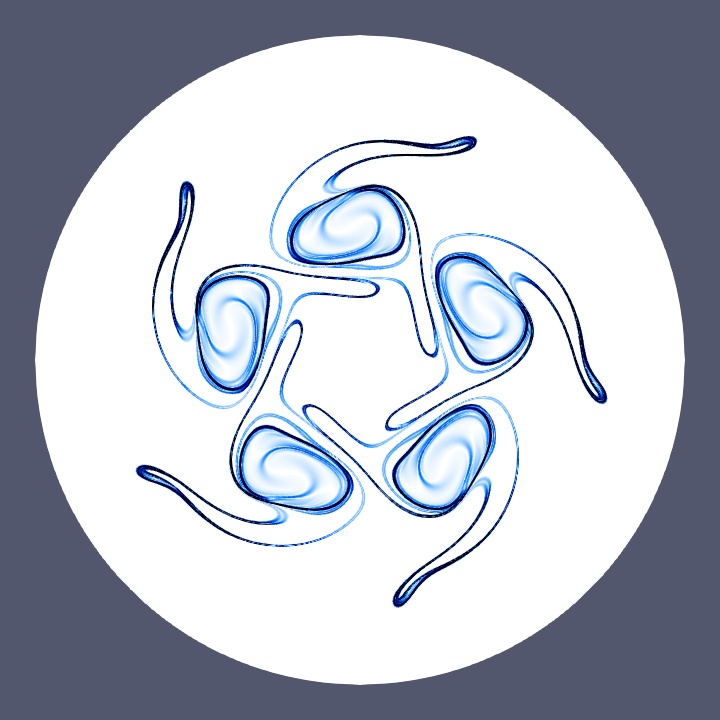}}}\;
    \subfloat[$t=7/8\,t_f$]{\fbox{\includegraphics[width=40mm]{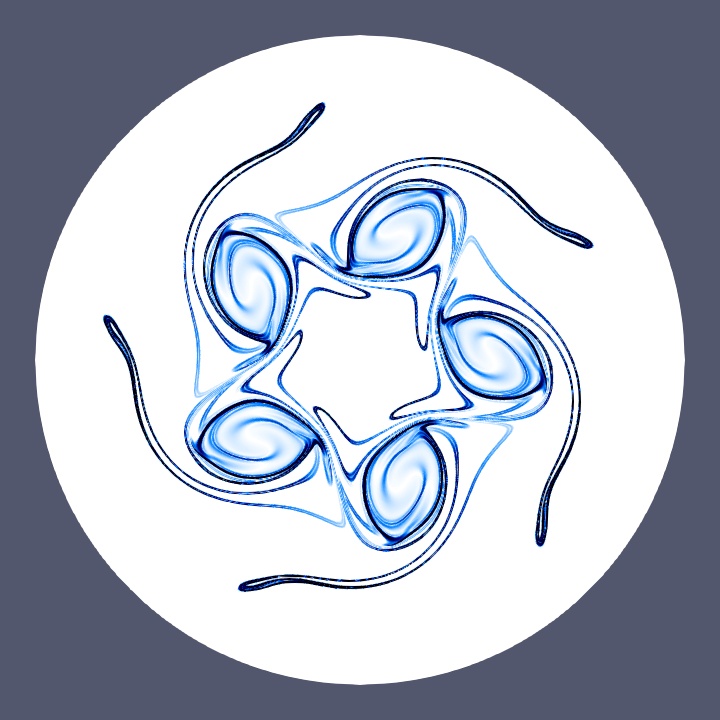}}}\;
    \subfloat[$t=t_f$]{\fbox{\includegraphics[width=40mm]{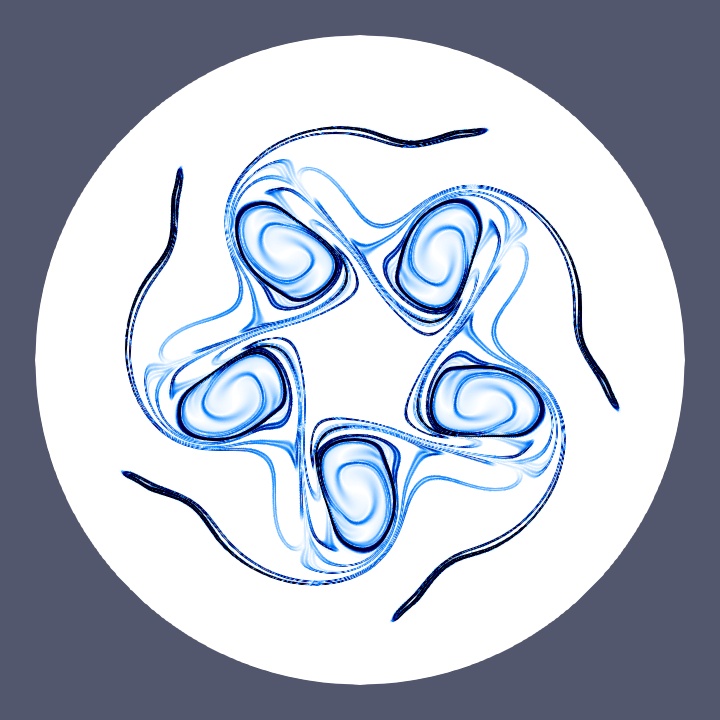}}}
  \end{center}
  \caption{\label{fig:5-mode}%
    Temporal snapshots of a schlieren plot of the density profile of the fifth
    mode diocotron instability test case. Reference computation with no restart
    on refinement level $r=9$ amounting, to 12,582,912 dG degrees of freedom
    per component. Here, $t_f = 10$.
  }
\end{figure}

\paragraph{Growth rate}
Next, we compare the  growth rate of each mode to the theoretical growth rate
predicted by the linear stability analysis conducted
in~\cite{davidson_influence_1998}. In the analysis, the electrostatic potential
$\varphi$ is dominated by the amplitude of its unstable $\ell$th Fourier mode,
which grows exponentially in time:
\begin{equation}
  \varphi(\vec x,t)
  \;\sim\;
  \exp\big(\textrm{i}(\ell \arctan(x_2/x_1) - \omega_\ell t)\big),
\end{equation}
where $\omega_\ell$ is a complex frequency, and
$\gamma_\ell:=\mathrm{Im}(\omega_\ell) > 0$ denotes the unstable growth rate of
the $\ell$th mode. The growth rate $\gamma_\ell$ can be computed explicitly in
terms of the order $\ell$ of the mode, the diocotron frequency
$\omega_{\mathrm{d}}$, and the geometric parameters $r_0$, $r_1$, and $R$; see
\cite[Equations~26, 27, and 28]{davidson_influence_1998}. For our setup, we
obtain the theoretical growth rates
\begin{align*}
  \gamma_3\approx 0.772, \qquad \gamma_4\approx 0.911, \qquad
  \gamma_5\approx 0.683.
\end{align*}
In order to compare our computational results to these theoretical
predictions, we follow a procedure outlined in~\cite[Section
7.5]{Crockatt2021} and \cite{Crockatt2022} to numerically estimate the
amplitude of the $\ell$th Fourier mode of our numerically computed
potential. At each time step $t_n$ we take a discrete Fourier transform of
the potential $\varphi_h^n$ interpreted as a function of polar angle
$\theta$ for fixed radius $r=r_0$, viz. $\varphi_h^n(\theta, r=r_0)$. Then,
we take the modulus of the coefficient corresponding to the $\ell$th mode.

Plots of the numerical amplitudes, normalized by their initial value, over time
for modes 3, 4, and 5 are shown in Figure~\ref{fig:diocotron-growth-rates}.
Since the stability analysis conducted in \cite{davidson_influence_1998} is
obtained from studying a small and linearized perturbation of a steady-state
solution, we can only hope for our numerical solution to agree with the
theoretical growth rate for a small time interval---after an initial startup
phase and before strong, nonlinear dynamics dominate. Therefore, for each plot
separately, we first visually compare the slope of the numerically computed
amplitudes to the slope of the (shifted) theoretical amplitudes
$\exp(\gamma_\ell t)$ and determine a region over which we average values for
the growth rates. The region are indicated in the plots with square brackets.
We then fit an exponential curve to the numerically computed amplitudes for
these time ranges. A table of the fitted numerical growth rates
$\gamma_{\ell,h}$ and their deviation from the theoretical growth rate
$\gamma_\ell$ is summarized in Figure~\ref{fig:diocotron-growth-rates}(d) for
the three modes.
\begin{figure}
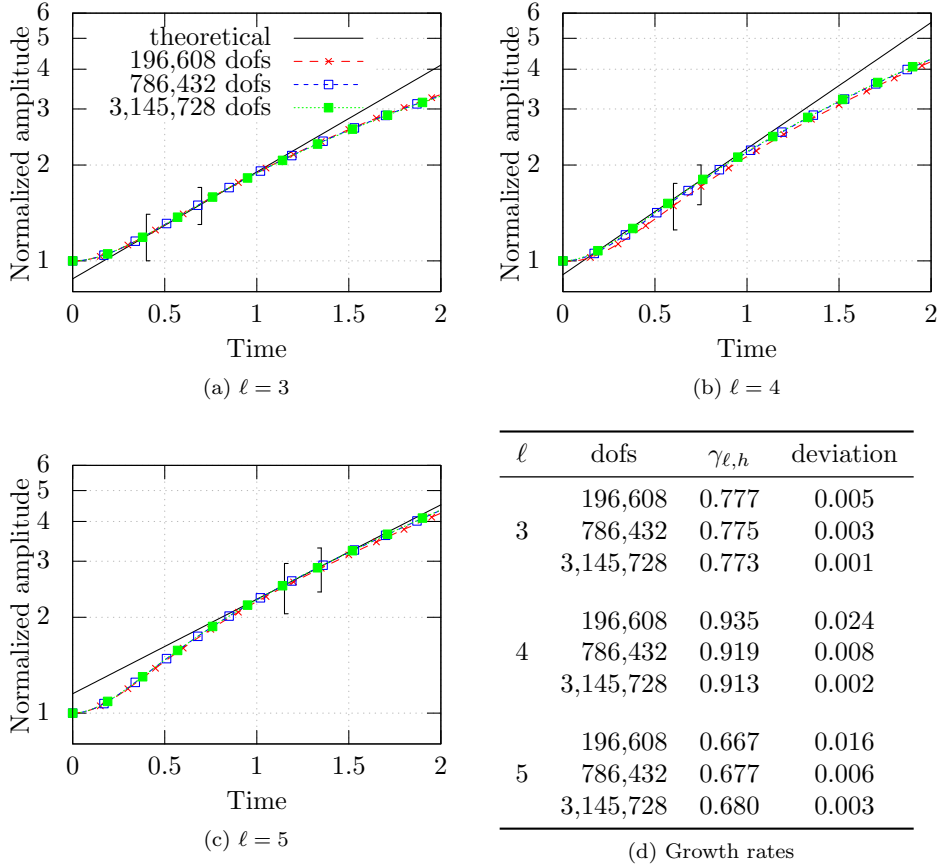
%
  \subfloat[$\ell = 3$]{%
    \input{diocotron_growth_rates-mode_3-amp-mode_3.tex}%
  }%
  \subfloat[$\ell = 4$]{%
    \input{diocotron_growth_rates-mode_4-amp-mode_4.tex}%
  }%
  \vspace{1em}
  \begin{minipage}{0.52\textwidth}%
    \subfloat[$\ell = 5$]{%
      \input{diocotron_growth_rates-mode_5-amp-mode_5.tex}%
    }%
  \end{minipage}%
  \hspace{-0.4em}%
  \begin{minipage}{0.49\textwidth}%
    \subfloat[Growth rates]{%
      \begin{tabular}{crcc}
        \toprule
        \multicolumn{1}{c}{$\ell$}&
        \multicolumn{1}{c}{dofs}&
        \multicolumn{1}{c}{$\gamma_{\ell,h}$}&
        \multicolumn{1}{c}{deviation}\\
        \midrule
        & 196,608     & 0.777 & 0.005 \\
        3 & 786,432     & 0.775 & 0.003 \\
        & 3,145,728   & 0.773 & 0.001 \\[1em]
        & 196,608     & 0.935 & 0.024 \\
        4 & 786,432     & 0.919 & 0.008 \\
        & 3,145,728   & 0.913 & 0.002 \\[1em]
        & 196,608     & 0.667 & 0.016 \\
        5 & 786,432     & 0.677 & 0.006 \\
        & 3,145,728   & 0.680 & 0.003 \\
        \bottomrule
      \end{tabular}
    }%
  \end{minipage}%
  \caption{\label{fig:diocotron-growth-rates}%
    Theoretical versus computed growth rates for modes (a) 3, (b) 4, and (c) 5.
    Log scale on the $y$ axis.
    Numerical growth rates are computed from the numerical amplitudes by
    fitting an exponential curve to the data between the square brackets in the
    plots from (a) $t = 0.4$ to $t = 0.7$, (b) $t = 0.6$ to $t = 0.75$, (c) $t
    = 1.15$ to $t = 1.35$. The numerical growth rates and their deviations from
    the theoretical growth rate (a) $\gamma_3 \approx 0.772$, (b) $\gamma_4
    \approx 0.911$, (c) $\gamma_5 \approx 0.683$ are given in (d).
  }%
\end{figure}

\paragraph{Gauß law residual and energy}
We compute the \emph{relative Gauß law residual} (see
Corollary~\ref{cor:gauss_law_residual})
\begin{equation*}
  \widehat{\mathcal R}_h^n := \frac{\|\nabla \widetilde \varphi_h^n -
  \nabla \varphi_h^n\|_{\Ltwod}}{\|\nabla \varphi_h^n\|_{\Ltwod}}
\end{equation*}
and the \emph{relative total energy}
\begin{equation*}
  \widehat {\mathcal E}_{\mathrm{tot},h}^n :=
  \frac
  {
    \sum_i \frac{m_i}{2} \rho_i^n |\vel_i^n|^2 +
    \frac{1}{2\alpha}\int_D |\nabla \varphi_h^n|^2\,dx
  }
  {
    \sum_i \frac{m_i}{2} \rho_i^0 |\vel_i^0|^2 +
    \frac{1}{2\alpha}\int_D |\nabla \varphi_h^0|^2\,dx
  }
\end{equation*}
at each time step $t_n$ for three mesh refinements and three restart
strategies: ``no restart", ``full restart"
(Definition~\ref{def:full_restart}), and ``relaxation"
(Definition~\ref{def:relaxation}). We plot our results in
Figure~\ref{fig:diocotron_residual_energy}.

The ``no restart" strategy violates the Gauß law but conserves energy.
Conversely, the ``full restart" strategy loses energy over time but preserves
the Gauß law. These violations grow over time, but decrease with mesh
refinement. It is not a priori guaranteed that restarting the potential lowers
the potential energy, so this observation is surprising. It also means that the
``relaxation" strategy falls back to being identical to the ``full restart"
strategy, since no relaxation of the kinetic energy actually occurs (see
Remark~\ref{rem:relaxation}).
\begin{figure}%
  \subfloat[49152 dofs]{\input{energy_and_residual-residual-1.tex}}%
  \subfloat[49152 dofs]{\input{energy_and_residual-total_energy-1.tex}}\\
  \subfloat[196608 dofs]{\input{energy_and_residual-residual-2.tex}}%
  \subfloat[196608 dofs]{\input{energy_and_residual-total_energy-2.tex}}\\
  \subfloat[786432 dofs]{\input{energy_and_residual-residual-3.tex}}%
  \subfloat[786432 dofs]{\input{energy_and_residual-total_energy-3.tex}}%
  \caption{\label{fig:diocotron_residual_energy}%
    Relative Gauß law residual and relative energy over time for refinement
    levels (a), (b) $r = 5$, (c), (d) $r = 6$, and (e), (f) $r = 7$ of the
    diocotron instability. Plots for the ``no restart", ``relaxation", and
    ``full restart" strategies at each refinement level. Here, $t_f = 10$.
    The ``full restart" and ``relaxation" plots are identical.}%
\end{figure}%


\section{Conclusion}
In this paper, we discussed a fully discrete numerical scheme for the
electrostatic Euler-Poisson equations with a given magnetic field. The
scheme uses an operator split to treat the Euler subsystem and source
subsystem separately. The Euler subsystem is treated explicitly with a
graph viscosity method that is invariant domain preserving and
conservative. The source subsystem, which couples the electrostatic
potential, the momentum, and the magnetic field, exhibits multiscale
phenomena that operate at various timescales spanning many orders of
magnitude. To address this issue, while also maintaining an energy balance,
we solved the source system implicitly with a PDE Schur complement that
reduces to solving a Poisson type problem at each time step. We proved that
the full numerical scheme preserves the desired structure of the PDE
system, and we demonstrated that the scheme performs well in the magnetic
drift limit without having to resolve high-frequency oscillations. For
future work, we will adapt the methods presented here to the full
Euler-Maxwell system where both the electrostatic potential and the
magnetic field are not given but instead computed self-consistently from
Maxwell's equations.


\section*{Acknowledgments}
JH and MM: acknowledge partial support by the National Science Foundation
under grant DMS-2045636 and by the Air Force Office of Scientific Research,
USAF, under grant/contract number FA9550-23-1-0007. IT: has been supported
by the Department of Energy under grant/contract LDRD-CIS-226834, the
National Science Foundation under grant DMS-2409841, and by a Simons Travel
Award MPS-TSM-00007151. JNS: This work was partially supported by the Laboratory
Directed Research and Development program at Sandia National Laboratories, and
by U.S. Department of Energy, Office of Science (SC), Office of Advanced
Scientific Computing Research's [Applied Mathematics / Computer Science /
Advanced Computing Technologies] Competitive Portfolios program. Sandia is a
multimission laboratory managed and operated by National Technology and
Engineering Solutions of Sandia LLC, a wholly owned subsidiary of Honeywell
International Inc. for the U.S. Department of Energy’s National Nuclear Security
Administration under contract DE-NA0003525.


\appendix

\section{Barotropic closures}\label{AppBaro}
We consider the case of the compressible Euler equations with barotropic closure.
In this case: the pressure is given by $p = p(\rho)$, the sound speed is
$c = p'(\rho)$, and the specific internal energy is given by
$e = e(\rho) = \int^\rho r^{-2}p(r) \mathrm d r$,
which leads to the following entropy-dissipation balance for the barotropic
Euler system:
\begin{align*}
  \frac{\mathrm d}{\mathrm dt}
  \int_{\domain}
  \frac{1}{2 \rho} |\mom|^2 + \rho e(\rho)
\mathrm{d}\xcoord
+ \int_{0}^t \Big\{ \int_{\bdry} \Big( \frac{1}{2 \rho} |\mom|^2 + \rho
e(\rho) + p(\rho) \Big) \vel\cdot\normal \, \mathrm{d}o_x \Big\} \mathrm{d} t \leq
0.
\end{align*}
This motivates the following definitions:
\begin{align}
\label{IsothermalEfluxPair}
\eta(\state) := \tfrac{1}{2 \rho}
|\mom|^2 + \rho e(\rho) \ \ \ \text{and} \ \ \
\mathbb{q}(\state) :=
\big( \tfrac{1}{2 \rho} |\mom|^2 + \rho e(\rho) + p(\rho) \big)\vel,
\end{align}
for the mathematical entropy and entropy-flux, respectively. For the specific
case of the isothermal closure $p(\rho) = \theta \rho$, the specific internal
energy takes the form $e = e(\rho) = \theta \ln \rho$.

\section{Hyperbolic solver}\label{sec:HypSolver}
This section provides a brief outline of the numerical methods used to solve
the compressible Euler subsystem for all the computational experiments advanced in Section
\ref{sec:numerical}. This section does not introduce any novel concept,
idea, or numerical scheme, and it is only provided for the sake of completeness.
The main ideas advanced in this section were originally developed in the
sequence of papers \cite{GuePo2016I, GuePo2016II, Guer2018, GuePoTom2019} and
revised, for the case of discontinuous spatial discretizations, in
\cite{euler25}. In this brief appendix, we limit ourselves to describe the
first-order scheme. For the implementation of higher order schemes we refer the
reader to \cite{euler25}.


The low-order scheme is obtained by using a first-order graph viscosity method
suggested first in \cite{GuePo2016I} while also handling the boundary
conditions as described in \cite{euler25}. Let $t_n$ be the current time,
$\tau_n$ the current time step, and advance in time by setting $t_{n+1} =
t_n + \tau_n$. Let ${\state}_h^n = \sum_{i \in \mathcal{V}} {\state}_i^n
\phi_i({\vec x})$ be the finite element approximation at time $t_n$. The first
order approximation at time $t_{n+1}$ is computed as
\begin{align}\label{eq:FO}
  \begin{split}
    m_i\frac{{\state}_i^{\low,n+1} - {\state}_i^n}{\tau_n}
    &+ \sum_{j \in \mathcal{I}(i)} \flux({\state}_j^n) \mathbf{c}_{ij}
    - d_{ij}^\low ({\state}_j^n - {\state}_i^n) \\
    &+ \flux(\state_i^{\bdry,n})\, \bc_{i}^{\bdry}
    - d_{i}^\low (\state_i^{\bdry,n} - \state_i^n)= \bzero ,
  \end{split}
\end{align}
where $m_i$ is the lumped mass entry corresponding to the shape function
$\phi_i(\xcoord)$, the vectors $\bc_{ij}\in \mathbb{R}^d$ and
$\bc_{i}^{\bdry}\in \mathbb{R}^d$ were defined in formulas (2.10)-(2.11) of
reference \cite{euler25}. The vector $\bc_{i}^{\bdry}$ is zero, unless the
corresponding collocation point $\xcoord_i$ lies on the boundary $\bdry$
\cite{euler25}. The set $\mathcal{I}(i)$ is the so-called stencil which is
defined as $\mathcal{I}(i) = \big\{ j \in \vertices \, | \, \mathbf{c}_{ij} \not
= \bzero \big\}$, $\flux({\state}) \in \mathbb{R}^{(d+2) \times d}$ is the usual
Euler flux $\flux({\state}) = [\mom^\transp, \rho^{-1}\mom
\mom^\transp +\mathbb{I}p, \rho^{-1}\mom^\transp(\totme + p)]$. For the case of
barotropic models $\flux({\state}) \in
\mathbb{R}^{(d+1) \times d}$ is defined as $\flux({\state}) = [\mom^\transp,
\rho^{-1}\mom \mom^\transp +\mathbb{I}p]$. The state $\state_i^{\bdry,n}$
corresponds with boundary data which is assumed to be known. The low-order
graph-viscosities $d_{ij}^{\low,n} > 0 $ and $d_{i}^{\low,n} > 0$ in
\eqref{eq:FO} are computed as:
\begin{align}
  \label{eq:dij_low_order}
  \begin{cases}
    \begin{aligned}
      d_{ij}^{\low,n}
      &:= |\bc_{ij}|_{\ell^2} \lambda_{\mathrm{max}}(\state_i^{n},
      \state_j^{n}, \normal_{ij}),
      &\text{where }&
      \normal_{ij}= \frac{\bc_{ij}}{|\bc_{ij}|_{\ell^2}},
      \quad\text{for }i\not=j,
      \\[0.25em]
      d_{i}^{\low,n}
      &:=
      |\bc_{i}^{\bdry}|_{\ell^2}
      \lambda_{\mathrm{max}}(\state_i^{n}, \state_i^{\bdry,n},
      \normal_{i}),
      &\text{where }&
      \normal_{i}=
      \frac{\bc_{i}^{\bdry}}{|\bc_{i}^{\bdry}|_{\ell^2}}. \\
      d_{ii}^{\low,n} &:=
      -\sum_{j \in \mathcal{I}(i),j\not=i}
      d_{ij}^{\low,n} \;-\; d_{i}^{\low,n} .
    \end{aligned}
  \end{cases}
\end{align}
Here $\lambda_{\max}({\state}_L, {\state}_R, {\vec n})$ is the maximum wave
speed of the one dimensional Riemann problem: $\partial_t {\state} + \partial_x
(\flux({\state}) {\vec n}) = 0$, where $x = {\vec x} \cdot {\vec n}$, with
initial condition: ${\state}(x,0) = {\state}_L = [\rho_L, \mom_L,
\totme_L]^\transp$ if $x<0$, and ${\state}(x,0) = {\state}_R = [\rho_R, \mom_R,
\totme_R]^\transp$ if $x\geq 0$. The maximum wavespeed of this Riemann problem
can be computed exactly \cite[Chap.~4]{Toro2009} for the co-volume equations of
state, however this comes at the expense of solving a nonlinear problem. In
theory and practice, any upper bound of the maximum wavespeed of the Riemann
problem could be used in formula \eqref{eq:dij_low_order} while still preserving
rigorous mathematical properties of the scheme \cite{GuePo2016I, GuePoTom2019}.
We will denote, generically, any estimate of the maximum wavespeed as
$\lambda^{\#}({\state}_L, {\state}_R, \normal)$. For all the computations
reported in this paper we use the following maximum wavespeed estimates.

For the case of the co-volume equation of state, the pressure is given by
$p(1-b\rho)=(\gamma-1) e \rho$ with $b \ge 0$.
In this case, we use
$\lambda^{\#}({\state}_L, {\state}_R, {\vec n})$ defined by:
\begin{align}
\label{LambdaCovol}
&\lambda^{\#}({\state}_L, {\state}_R, {\vec n}) =
\max((\lambda_1^-(p^{\#}) )_{-},\ (\lambda_3^+(p^{\#}))_+), \\
\nonumber
&\lambda_1^-(p^{\#})=v_L - c_L\left(
1+\frac{\gamma+1}{2\gamma}\left(\frac{p^{\#}-p_L}{p_L}\right)_+
\right)^\frac12, \\
\nonumber
&\lambda_3^+(p^{\#})=v_R + c_R\left(
1+\frac{\gamma+1}{2\gamma}\left(\frac{p^{\#}-p_R}{p_R}\right)_+
\right)^\frac12, \\
\label{PSharp}
&p^{\#} := \left(
\frac{c_L(1-b\rho_L)+c_R(1-b\rho_R)-\frac{\gamma-1}{2}(v_R-v_L)}
{ c_L(1-b\rho_L)\ p_L^{-\frac{\gamma-1}{2\gamma}}
+ c_R(1-b\rho_R)\ p_R^{-\frac{\gamma-1}{2\gamma} } }
\right)^{\frac{2\gamma}{\gamma-1}},
\end{align}
where $z_-:=\max(0,-z)$, $z_+:=\max(0,z)$, $v_L = \vel_L \cdot \normal$,
$v_R = \vel_R \cdot \normal$, $p_L$ and $p_R$ are the left and right pressures,
and $c_L$ and $c_R$ are left and right sound speeds. Formula \eqref{PSharp} is
often referred to as the two-rarefaction estimate \cite{Toro2009}. It is
possible to show that $\lambda^{\#}({\state}_L,
{\state}_R, \normal) \geq \lambda_{\max}({\state}_L,
{\state}_R, \normal)$ for $1 < \gamma \le \frac53$, see \cite{GuePo2016II}. We
finally mention that scheme \eqref{eq:FO} equipped with
the viscosity \eqref{LambdaCovol} guarantees the entropy-dissipation assumption
\eqref{EulerSolverDissipation}.

If the pressure is independent of the specific internal energy, that is $p =
p(\rho)$, and $p'(\rho) \geq 0$, we have that admissibility of the scheme only
requires positivity of the density. In this context, the sound speed is simply
given by $c = \sqrt{p'(\rho)}$. It can be proven that any choice of
$\lambda^{\#}({\state}_L, {\state}_R, \normal)$
satisfying the property
$
\lambda^{\#}({\state}_L, {\state}_R, \normal)
\geq
\max\{|\vel_L\cdot\normal|, |\vel_R\cdot\normal|\}
$
is enough to guarantee
positivity of the density.
However, this is, in practice, not enough to tame the behaviour of non-smooth
solutions.
Therefore, for the case of a barotropic equation of state, we use
\begin{align*}
\lambda^{\#}({\state}_L, {\state}_R, \normal)
=
\max\{|\vel_L\cdot\normal|, |\vel_R\cdot\normal|\} + \max\{c_L, c_R\},
\end{align*}
where $c_L = \sqrt{p'(\rho_L)}$ and $c_R = \sqrt{p'(\rho_R)}$.
Note that this estimate of the maximum wavespeed is more than enough to guarantee
admissibility of the solution. However, we make no claim as to whether this
estimate is a sharp upper bound of the maximum wavespeed.

\begin{remark}[Convex reformulation and CFL condition] The scheme
\eqref{eq:FO} can be rewritten as
\begin{multline}
\label{ConvexRef}
\state_i^{\low,n+1} = \left(1 +\frac{2 \dt_n
d_{ii}^{\low,n}}{m_i}\right)
\state_{i}^{n}
+ \frac{2 \dt_n d_{i}^{\low,n}}{m_i} \overline{\state}_{i}^{\bdry,n}
+ \sum_{j \in \mathcal{I}(i)\backslash\{i\}} \frac{2
\dt_n d_{ij}^{\low,n}}{m_i} \overline{\state}_{ij}^{n},
\end{multline}
where $\overline{\state}_{ij}^{n}$ and $\overline{\state}_{i}^{\bdry,n}$ are
the bar-states defined by
\begin{align}
  \label{UsualBarState}
  \overline{\state}_{ij}^{n} &= \frac{1}{2}(\state_j^{n} +
  \state_i^{n})
  - \frac{|\bc_{ij}|}{2 d_{ij}^\low} \left(\flux(\state_j^{n}) -
  \flux(\state_i^{n})\right) \normal_{ij}, \\
\label{BoundaryBarState}
\overline{\state}_{i}^{\bdry,n} &=
\frac{1}{2}(\state_i^{\bdry,n} + \state_i^{n})
- \frac{|\bc_i^{\bdry}|}{2 d_{i}^\low}
\left(\flux(\state_i^{\bdry,n}) -
\flux(\state_i^{n})\right)
\normal_{i}.
\end{align}
We note that the states $\{\overline{{\state}}_{ij}^{n}\}_{j \in
\mathcal{I}(i)}$ are admissible, provided that ${\state}_i^{n}$ and ${\vec
u}_j^{n}$ are admissible \cite{GuePo2016I, GuePoTom2019, euler25}.
We also note that ${\state}_i^{n+1}$ is a convex combination of the bar-states
$\{\overline{{\state}}_{ij}^{n}\}_{j \in \mathcal{I}(i)}$ and
$\overline{\state}_{i}^{\bdry,n} $ provided the condition
$\big(1 +\frac{2 \dt_n d_{ii}^{\low,n}}{m_i} \big) \geq 0$ holds.
Therefore, we define the largest admissible time-step size as
\begin{align}
\label{CflFormula}
  \dt_n = \mathrm{CFL} \cdot \min_{i \in \vertices}
  \Big(-\frac{m_i}{2 d_{ii}^{\low,n}}\Big)
\end{align}
where $\mathrm{CFL} \in (0,1)$ is a user defined parameter.
\end{remark}

Describing the extension of this scheme to higher order accuracy is beyond the
scope of this paper.
Here, we briefly mention that the second-order accurate scheme used for all our
computations uses (exactly) the flux-limited high-order method described in
\cite{euler25}.
For the case of the barotropic Euler system, we implemented the
entropy-viscosity method as described in \cite{euler25} using the entropy-flux
pair defined in \eqref{IsothermalEfluxPair}.

\bibliographystyle{siamplain}
\bibliography{references}

\end{document}